\documentclass[11 pt,english]{article}
\usepackage[activeacute]{babel}
\usepackage{amsfonts} 
\usepackage{mathrsfs} 
\usepackage{amssymb, amsmath, amsfonts}
\usepackage{makeidx}
\usepackage{epsfig}
\usepackage{color}\usepackage[labelfont=normal]{subcaption}
\usepackage[T1]{fontenc}
\usepackage{graphics,graphicx,graphpap}
\usepackage[ansinew]{inputenc}
\usepackage{amsthm}
\usepackage{multirow}

\newcommand\blfootnote[1]{%
  \begingroup
  \renewcommand\thefootnote{}\footnote{#1}%
  \addtocounter{footnote}{-1}%
  \endgroup
}

\if@twoside \oddsidemargin 5pt \evensidemargin 5pt \marginparwidth 20pt
 \else \oddsidemargin 5pt \evensidemargin 5pt \marginparwidth 20pt \fi

\newcommand{\ZZ}{\mathbb{Z}}

\newcommand{\Aut}{\mathrm{Aut}}
\newcommand{\G}{\Gamma}

\newcommand{\CPM}{\mathop{\rm CPM}}
\newcommand{\CPMall}{\overline{\mathop{\rm CPM}}}
\newcommand{\mb}{\mathbf}

\newcommand{\la}{\langle}
\newcommand{\ra}{\rangle}
\newcommand{\rad}{\mathrm{rad}}
\newcommand{\att}{\mathrm{att}}
\newcommand{\Gall}{\bar{G}}
\newcommand{\Gmall}{\bar{\G}}
\newcommand{\Rall}{\bar{R}}
\newcommand{\Tall}{\bar{T}}
\newcommand{\Vall}{\bar{V}}
\newcommand{\rhoall}{\bar{\rho}}
\newcommand{\tauall}{\bar{\tau}}
\newcommand{\sigmaall}{\bar{\sigma}}
\newcommand{\etaall}{\bar{\eta}}
\newcommand{\nuall}{\bar{\nu}}
\newcommand{\lea}{\langle}
\newcommand{\ria}{\rangle}
\newcommand{\laa}{\langle}
\newcommand{\raa}{\rangle}

\newtheorem{theorem}{Theorem}[section]
\newtheorem{proposition}[theorem]{Proposition}
\newtheorem{corollary}[theorem]{Corollary}
\newtheorem{lemma}[theorem]{Lemma}

\newtheorem{question}[theorem]{Question}

\theoremstyle{definition}
\newtheorem*{remark}{Remark}
\newtheorem{construction}[theorem]{Construction}
\newtheorem{assumption}[theorem]{Assumption}

\begin{document}

\begin{center}
\Large{\textbf{Generalized Gardiner-Praeger graphs and their symmetries}} \\ [+4ex]
\v Stefko Miklavi\v c{\small$^{a, b, c}$}, Primo\v z \v Sparl{\small$^{a, c, d,*}$}, Stephen E.~Wilson{\small$^{e}$}
\\ [+2ex]
{\it \small 
$^a$University of Primorska, Institute Andrej Maru\v si\v c, Koper, Slovenia\\
$^b$University of Primorska, FAMNIT, Koper, Slovenia\\
$^c$Institute of Mathematics, Physics and Mechanics, Ljubljana, Slovenia\\ 
$^d$University of Ljubljana, Faculty of Education, Ljubljana, Slovenia\\
$^e$Northern Arizona University, Flagstaff, Arizona, USA
}
\end{center}

\blfootnote{
Email addresses: 
stefko.miklavic@upr.si (\v Stefko Miklavi\v c),
primoz.sparl@pef.uni-lj.si (Primo\v z \v Sparl),
stephen.wilson@nau.edu (Stephen E. Wilson)\\
* - corresponding author
}


\hrule

\begin{abstract}
A subgroup of the automorphism group of a graph acts {\em half-arc-transitively} on the graph if it acts transitively on the vertex-set and on the edge-set of the graph but not on the arc-set of the graph. If the full automorphism group of the graph acts half-arc-transitively, the graph is said to be {\em half-arc-transitive}. 

In 1994 Gardiner and Praeger introduced two families of tetravalent arc-transitive graphs, called the $C^{\pm 1}$ and the $C^{\pm \varepsilon}$ graphs, that play a prominent role in the characterization of the tetravalent graphs admitting an arc-transitive group of automorphisms with a normal elementary abelian subgroup such that the corresponding quotient graph is a cycle. All of the Gardiner-Praeger graphs are arc-transitive but admit a half-arc-transitive group of automorphisms. Quite recently, Poto\v cnik and Wilson introduced the family of $\CPM$ graphs, which are generalizations of the Gardiner-Praeger graphs. Most of these graphs are arc-transitive, but some of them are half-arc-transitive. In fact, at least up to order $1000$, each tetravalent half-arc-transitive loosely-attached graph of odd radius having vertex-stabilizers of order greater than $2$ is isomorphic to a $\CPM$ graph.

In this paper we determine the automorphism group of the $\CPM$ graphs and investigate isomorphisms between them. Moreover, we determine which of these graphs are $2$-arc-transitive, which are arc-transitive but not $2$-arc-transitive, and which are half-arc-transitive.
\end{abstract}

\hrule

\begin{quotation}
\noindent {\em \small Keywords: tetravalent; $\CPM$ graph; symmetry; automorphism; arc-transitive; half-arc-transitive}
\end{quotation}

\section{Introduction}
\label{sec:Intro}

Investigation of symmetries of graphs has been a very active topic of research for decades. In particular, the examples for which some subgroup $G$ of automorphisms acts transitively on the vertex-set of the graph (such a graph is said to be {\em $G$-vertex-transitive}) and at the same time acts transitively on the edge-set of the graph (and so the graph is also {\em $G$-edge-transitive}) have received much attention. If this group $G$ also acts transitively on the set of all ordered pairs of adjacent vertices (called {\em arcs}) the graph is said to be {\em $G$-arc-transitive} (or {\em $G$-symmetric}). While it is not difficult to see that a cubic graph admitting a vertex- and edge-transitive group of automorphisms $G$ is automatically also $G$-arc-transitive, this is not the case when one considers tetravalent graphs. In other words, there exist tetravalent graphs admitting a group of automorphisms $G$ such that the graph is $G$-vertex-transitive and $G$-edge-transitive, but not $G$-arc-transitive. In such a case this graph is said to be {\em $G$-half-arc-transitive}. 

Numerous papers on tetravalent graphs admitting a vertex- and edge-transitive or even arc-transitive group of automorphisms have been published, by far too many to mention all of them here (but see for instance~\cite{KuzMalPot18, PotWil??, RamSpa19} and the references therein). Mentioning a few of them might serve as a motivation for our topic. In 1994 Gardiner and Praeger~\cite{GarPra94A, GarPra94} initiated the now very well known ``normal quotients'' approach to the investigation of tetravalent graphs admitting an arc-transitive group, say $G$, of automorphisms. In particular, they investigated examples for which the group $G$ has an elementary abelian normal subgroup, say $N$. In~\cite{GarPra94} they focused on the situation in which the quotient graph with respect to the orbits of $N$ is a cycle. Doing so, they introduced two new families of graphs, called the $C^{\pm 1}$ and the $C^{\pm \epsilon}$ graphs. These graphs do not cover all of the examples having the above described property but quite recently, Kuzman, Malni\v c and Poto\v cnik~\cite{KuzMalPot18} managed to complete the characterization initiated in~\cite{GarPra94}. Moreover, they generalized the results of~\cite{GarPra94} by considering not only graphs admitting arc-transitive groups with elementary abelian normal subgroups giving rise to a cycle quotient graph, but also the ones admitting a half-arc-transitive group with this property.

Given that the tetravalent vertex- and edge-transitive graphs are widely studied it is not surprising that various censi of such graphs have been constructed. For instance, Poto\v cnik, Spiga and Verret constructed the census of all tetravalent graphs up to order $1000$ that admit a half-arc-transitive group of automorphisms~\cite{PotSpiVer15}. Similarly, Poto\v cnik and Wilson constructed the census of all known tetravalent edge-transitive graphs up to order $512$~\cite{PotWil??} (this census is potentially not complete). Within it, the authors introduce a new family of graphs, called the {\em CPM graphs}, which are generalizations of the above mentioned $C^{\pm 1}$ and $C^{\pm \epsilon}$ graphs. One of the main aspects of this generalization is that for these graphs the normal subgroup giving rise to the cycle quotient graph need not be elementary abelian. But what makes this generalization really interesting is that while all of the $C^{\pm 1}$ and $C^{\pm \epsilon}$ graphs are arc-transitive, infinitely many $\CPM$ graphs are half-arc-transitive (by which we mean that the full automorphism group of the graph is half-arc-transitive). In this sense the $\CPM$ graphs are also related to the graphs from~\cite{KuzMalPot18}, but as mentioned, there the normal subgroup is again elementary abelian. 
\medskip

The above mentioned connections of the $\CPM$ graphs to the results of~\cite{GarPra94} and~\cite{KuzMalPot18} are good enough reasons for the investigation of this interesting family of graphs. But there is another very important aspect of the fact that some of them are actually half-arc-transitive, which makes their investigation much more intriguing. To be able to explain it we first need to review some terminology and results from the theory of half-arc-transitive graphs. 

Graphs admitting a half-arc-transitive group of automorphisms have been widely studied in the last three decades with the first few results focusing on constructions of such graphs with various additional properties (see for instance~\cite{MalMar99, TayXu94}) or on classifications of such graphs of some restricted orders (see for instance~\cite{AlsXu94, Xu92}). To this day the vast majority of results on graphs admitting a half-arc-transitive group of automorphisms and on half-arc-transitive graphs themselves focus on the tetravalent graphs. 

An important step forward in the investigation of such graphs was made in 1998 when Maru\v si\v c~\cite{Mar98} proposed a method for the investigation of the local structure of such graphs that proved to be very fruitful. We briefly describe the main idea (but see~\cite{Mar98} for more details). Let $\G$ be a tetravalent graph admitting a half-arc-transitive group $G \leq \Aut(\G)$. The action of $G$ on $\G$ then gives rise to two paired orientations of the edges of $\G$. In each of these two oriented graphs the in-valence and out-valence are both equal to $2$. This gives rise to {\em $G$-alternating cycles} of $\G$ which are simply the cycles of $\G$ corresponding to the alternating cycles of any of the above two oriented graphs (cycles on which any two consecutive edges are oppositely oriented). Half of the length of these cycles is called the {\em $G$-radius} of $\G$ and is denoted by $\rad_G(\G)$ and the size of the intersection of any two non-disjoint $G$-alternating cycles is called the {\em $G$-attachment number} and is denoted by $\att_G(\G)$. In the case that $\rad_G(\G) = \att_G(\G)$ the graph $\G$ is said to be {\em tightly $G$-attached}. At the other extreme, if $\att_G(\G) = 2$ or $\att_G(\G) = 1$, respectively, the graph $\G$ is said to be {\em antipodally $G$-attached} or {\em loosely $G$-attached}, respectively. 

The importance of these three situations stems from a result of Maru\v si\v c and Praeger~\cite{MarPra99} and even more so from a recent improvement by Ramos Rivera and \v Sparl~\cite{RamSpa19}. Namely, in~\cite{RamSpa19} a further refinement in the study of the local structure of tetravalent graphs admitting a half-arc-transitive group of automorphisms via the alternating cycles was proposed. By introducing the notion of the alternating jump it was proved that (with the exception of certain well known Cayley graphs of cyclic groups) every tetravalent graph $\G$ admitting a half-arc-transitive group of automorphisms $G$ is either tightly $G$-attached or one can find a cyclic normal subgroup $K$ of $G$ such that the corresponding quotient graph of $\G$ with respect to the orbits of $K$ is a simple tetravalent graph admitting a half-arc-transitive action of the quotient group $G/K$ relative to which the graph is either antipodally or loosely attached. In fact, the results of~\cite{RamSpa19} seem to indicate that, at least when one restricts to the half-arc-transitive graphs, one need not worry about the antipodally attached examples. The fact that the tightly attached tetravalent graphs are classified (see~\cite{Mar98, MarPra99, Spa08, Wil04}) thus calls for a thorough investigation of the loosely attached graphs. 

There is another reason why the loosely attached tetravalent half-arc-transitive graphs should be investigated. In~\cite{AlbAlkMutPraSpi16} the development of the ``normal quotients method'' for the investigation of tetravalent graphs admitting a half-arc-transitive group of automorphisms was initiated. By the above mentioned results from~\cite{RamSpa19} all so-called basic pairs from~\cite{AlbAlkMutPraSpi16} correspond to loosely or possibly antipodally attached graphs. To be able to classify or at least characterize the basic pairs we will thus necessarily have to get a better understanding of the loosely attached graphs.

This finally brings us back to the above mentioned censi from~\cite{PotSpiVer15} and~\cite{PotWil??}. The census from~\cite{PotSpiVer15} reveals that there are 3247 connected tetravalent half-arc-transitive graphs up to order $1000$ but only 20 of them have vertex-stabilizers (in the full automorphism group) which are not isomorphic to $\ZZ_2$ (see for instance~\cite{ConPotSpa15, Mar05, MarNed01} for some results on vertex-stabilizers in tetravalent half-arc-transitive graphs). What is more, all $20$ of these are loosely attached but only five of them have odd radius. Curiously enough, these five all belong to the family of $\CPM$ graphs (see Section~\ref{sec:CPM} for the definition). In particular, they are $\CPM(3,2,7;2)$, $\CPM(3,2,9;2)$, $\CPM(6,2,7;2)$, $\CPM(9,2,7;2)$ and $\CPM(6,2,9;2)$, and so the first two appear also in the census from~\cite{PotWil??}. In addition, the graph $\CPM(3,2,7;2)$, together with its full automorphism group, constitutes a basic pair in the sense of~\cite{AlbAlkMutPraSpi16}. 
\medskip

All of this motivates the investigation of symmetries of the $\CPM$ graphs, which is the main theme of this paper. In particular, we determine the full automorphism group of each $\CPM$ graph (see Theorem~\ref{the:Aut2AT}, Theorem~\ref{the:HAT} and Corollary~\ref{cor:Aut_AT}), show that the $\CPM$ graphs cannot be $3$-arc-transitive (an $s$-arc is a sequence $(v_0, v_1, \ldots , v_s)$ of vertices of the graph such that $v_{i-1} \neq v_{i+1}$ for all $1 \leq i \leq s-1$ and $v_i \sim v_{i+1}$ for all $0 \leq i \leq s-1$) and determine which of them are $2$-arc-transitive, which are arc-transitive but not $2$-arc-transitive, and which are half-arc-transitive (see Theorem~\ref{the:2ATclass} and Theorem~\ref{the:HAT}). We also determine almost all possible isomorphisms between $\CPM$ graphs (see Proposition~\ref{pro:all_iso_2AT} and Proposition~\ref{pro:alliso}).

\section{Notational conventions}
\label{sec:prelim}

Throughout the paper all graphs are assumed to be finite and simple. All considered graphs will be undirected, although we will often be working with an implicit orientation of their edges given by a half-arc-transitive action of a (sub)group of automorphisms of the graph under consideration. 

The residue class ring of integers modulo $n$ will be denoted by $\ZZ_n$ and its group of units by $\ZZ_n^*$. In many of our arguments we will be working with elements from $\ZZ_n$ and equations involving such elements. Such equations should always be considered as equations in $\ZZ_n$. We will quite often write things such as $r^{t} \pm 1 = 0$ for an element $r \in \ZZ_n$ and an integer $t$. By this we mean that (at least) one of $r^{t} + 1 = 0$ or $r^{t} - 1 = 0$ holds in the ring $\ZZ_n$. Sometimes we will also write things such as $r < t$ where $r \in \ZZ_n$ and $t$ is an integer. By this we mean that the smallest nonnegative member of the residue class of $r$ is smaller than $t$. 

When $m$ divides $n$ and $a \in \ZZ_n$, there is a unique $a' \in \ZZ_m$ such that $a' \equiv a \pmod{m}$.  In such cases, we will abbreviate that by saying ``Let $a'$ be $a$ mod $m$''.

For a subset $U$ of the vertex set $V(\G)$ of a graph $\G$ we let $\G[U]$ denote the subgraph of $\G$ induced by $U$. Similarly, for a pair of disjoint subsets $U_1$, $U_2$ of $V(\G)$ we let $\G[U_1, U_2]$ be the bipartite subgraph of $\G$ with vertex set $U_1 \cup U_2$ consisting of all the edges of $\G$ with one end-vertex in $U_1$ and the other in $U_2$. Finally, for two sequences $W_1$, $W_2$ of vertices of $\G$ (usually this will be walks) we denote their concatenation by $W_1 \cdot W_2$.

\section{The $\CPM$ graphs}
\label{sec:CPM}

In this section we start our investigation of the $\CPM$ graphs. They were first introduced by Poto\v cnik and Wilson in~\cite{PotWil??} and are a natural generalization of the $C^{\pm 1}$ and $C^{\pm \epsilon}$ graphs of Gardiner and Praeger~\cite{GarPra94} on one hand and of the $\mathcal{X}_o(m,n;r)$ graphs from~\cite{Mar98, Spa08} or equivalently of the power spider graphs $\mathrm{PS}(m,n,r)$ from~\cite{Wil04} on the other hand. 

\begin{construction}
Let $m,s,n$ be positive integers with $n \geq 3$ and let $r \in \ZZ_n^*$ be such that $r^{ms} \pm 1 = 0$. The graph $\CPMall(m,s,n;r)$ is then the graph with vertex set consisting of all pairs $\laa i; \mb{v}\raa$, where $i \in \ZZ_{ms}$ and $\mb{v} \in \ZZ_n^s$, in which adjacency is defined by the following rule:
$$
	\laa i;\mb{v}\raa \sim \laa i+1; \mb{v} \pm r^i \mb{e_\ell}\raa,
$$ 
where $\ell$ is $i$ mod $s$ and $\mb{e_\ell} \in \ZZ_n^s$ is the $\ell$-th standard vector with the understanding that each $\mb{v} \in \ZZ_n^s$ is viewed as $\mb{v} = (v_0,v_1, \ldots , v_{s-1})$, and so we refer to $v_0$ as the $0$-th component of $\mb{v}$ and to $v_{s-1}$ as the $(s-1)$-th component of $\mb{v}$. The graph $\CPM(m,s,n;r)$ is then defined as the connected component of $\CPMall(m,s,n;r)$, containing the vertex $\laa 0;\mb{0}\raa$. 
\end{construction}

Note that since $i \in \ZZ_{ms}$, the addition $i+1$ is computed modulo $ms$ and similarly $\mb{v} \pm r^i \mb{e_\ell}$ is computed component-wise and modulo $n$. We make the following notational convention. Whenever there will be the need to spell out the components of the $s$-tuple $\mb{v}$ of $\laa i; \mb{v}\raa$ we will write $\laa i; (v_0, v_1, \ldots , v_{s-1})\raa$.
We also introduce the following notation. For each $i \in \ZZ_{ms}$ we let
\begin{equation}
\label{eq:V_i}
	\Vall_i = \{\laa i ; \mb{v}\raa \colon \mb{v} \in \ZZ_n^s\}\quad \text{and}\quad V_i = \Vall_i \cap V(\CPM(m,s,n;r)).
\end{equation}
Of course, the $ms$ sets $\Vall_i$ partition the vertex-set of $\CPMall(m,s,n;r)$ and the $ms$ sets $V_i$ partition the vertex-set of $\CPM(m,s,n;r)$. To illustrate the construction of $\CPM$ graphs and the role of the sets $V_i$ in them, a part of the graph $\CPM(7,2,5;2)$ is depicted on Figure~\ref{fig:example}.
\begin{figure}[h]
\begin{center}
	\includegraphics[scale=0.45]{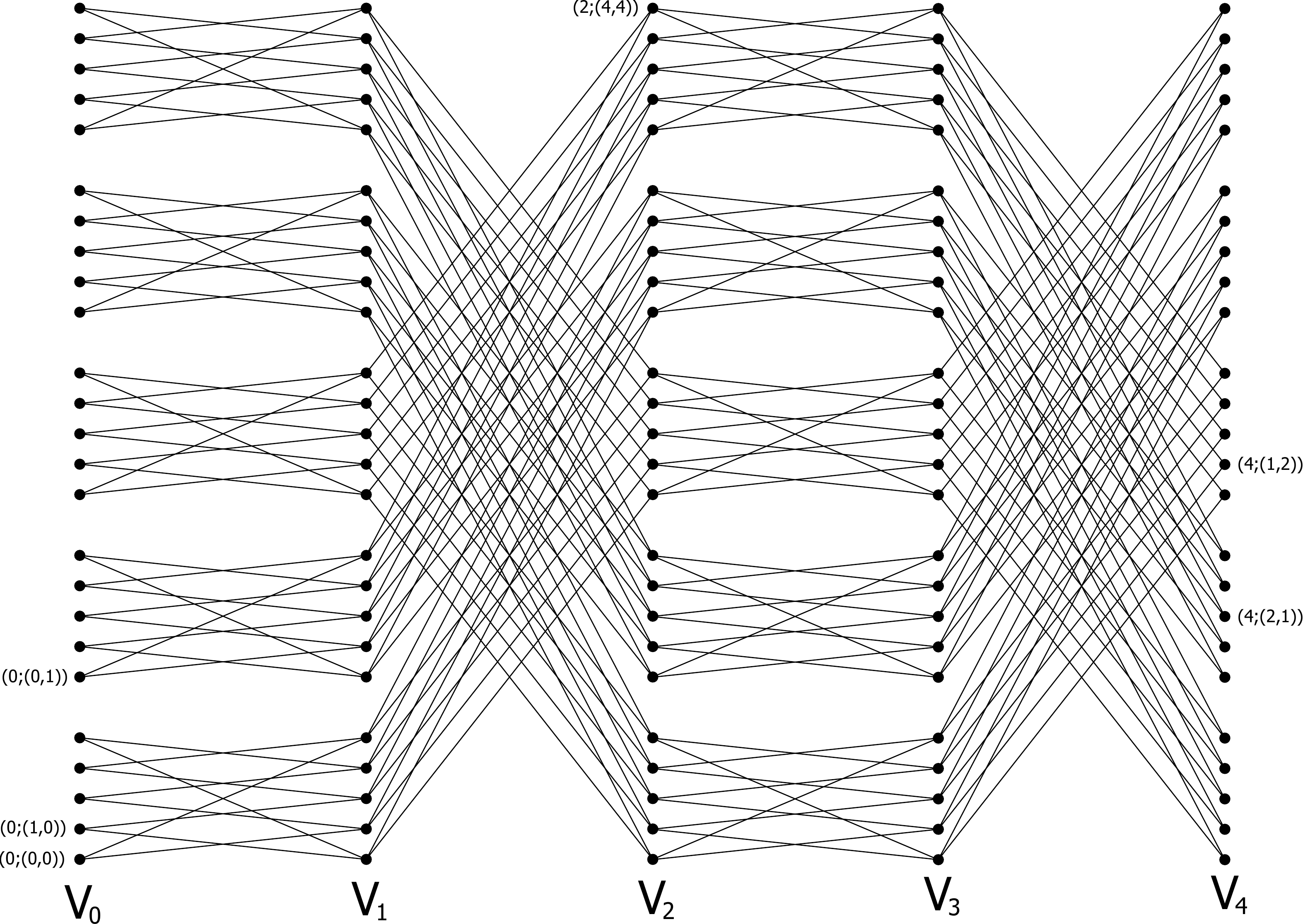}
	\caption{A section of the graph $\CPM(7,2,5;2)$.}
	\label{fig:example}
\end{center}
\end{figure}

\begin{remark}
A few remarks are in order. The first is about notation. In~\cite{PotWil??} the order of the parameters in the definition is slightly different (our $\CPM(m,s,n;r)$ would be $\CPM(n,s,m,r)$ there) and similarly the vertices of these graphs are denoted by pairs $(\mb{v},i)$ instead of $\laa i;\mb{v}\raa$ in~\cite{PotWil??}. However, to be consistent with the notation for the $\mathcal{X}_o(m,n;r)$ graphs from~\cite{Spa08} and of the Power spider graphs $\mathrm{PS}(m,n,r)$ and the way their vertices are denoted in~\cite{Wil04}, we decided to make this small change. This notation is also consistent with~\cite{RamSpa17} which deals with similar graphs of higher valences and in which some of the arguments for determining their symmetries are quite similar to the ones we are using in this paper. We also point out that in~\cite{PotWil??} the authors did not use a different notation for the whole graph $\CPMall(m,s,n;r)$ in the case that it is not connected, but simply worked with the component $\CPM(m,s,n;r)$. Nevertheless, as it may sometimes be convenient to refer to the whole $\CPMall(m,s,n;r)$, which is always of order $msn^s$, we think it is convenient to introduce this special notation for it.

The second remark is about connectedness. As was pointed out in~\cite{PotWil??} it is easy to see that the graph $\CPMall(m,s,n;r)$ is connected (in which case $\CPMall(m,s,n;r) = \CPM(m,s,n;r)$) if and only if $n$ is odd. Moreover, in the case that $n$ is even each connected component of $\CPMall(m,s,n;r)$ (and thus also the graph $\CPM(m,s,n;r)$) has $ms(n/2)^s$ vertices in the case that $m$ is even, and has $2ms(n/2)^s$ vertices in the case that $m$ is odd. More precisely, in the case that $n$ is even the vertex-set of the graph $\CPM(m,s,n;r)$ can be described as follows. If $m$ is even, $V_0$ consists of all the vertices of the form $\laa 0;(v_0,v_1, \ldots , v_{s-1})\raa$ with $v_i$ all even, $V_1$ consists of all the vertices of the form $\laa 1;(v_0,v_1, \ldots , v_{s-1})\raa$ with $v_0$ odd and all other $v_i$ even, etc. If however $m$ is odd, then $V_0$ consists of all the vertices of the form $\laa 0;(v_0,v_1, \ldots , v_{s-1})\raa$ where either all $v_i$ are even or all are odd, $V_1$ consists of all the vertices of the form $\laa 1;(v_0,v_1, \ldots , v_{s-1})\raa$ where either $v_0$ is odd and all other $v_i$ are even, or $v_0$ is even and all other $v_i$ are odd, etc. 

The last remark is about the parameter $s$. In~\cite{PotWil??} it is assumed that $s \geq 2$. The only reason we have decided to allow $s = 1$ in our definition is that in the case of $s = 1$ the corresponding graph $\CPM(m,1,n;r)$ is in fact the graph $\mathrm{PS}(m,n,r)$ which thus also reveals why the $\CPM$ graphs are indeed generalizations of the Power spider graphs. Nevertheless, since the Power spider graphs belong to the well studied family of tetravalent graphs admitting a half-arc-transitive group of automorphisms relative to which they are tightly attached (see~\cite{Mar98, MarPra99, Spa08, Wil04}) we will mainly be concerned only with the graphs with $s \geq 2$.
\end{remark}

Before we start a thorough analysis of symmetries of the $\CPM$ graphs we record some fairly obvious isomorphisms between the graphs $\CPM(m,s,n;r)$ in the case that the whole graph $\CPMall(m,s,n;r)$ is not connected.

\begin{lemma}
\label{le:discon}
Let $m, s, n$ be positive integers with $n \geq 2$ and let $r \in \ZZ_{2n}^*$ be such that $r^{ms} = \pm 1$. Then the following hold:
\begin{itemize}\itemsep = 0pt
\item[(i)] if $m$ is odd then $\CPM(2m,s,2n;r) \cong \CPM(m,s,2n;r)$;
\item[(ii)] if $n$ is odd then, depending on whether $m$ is even or odd, we have that $\CPM(m,s,2n;r) \cong \CPM(m,s,n;r')$ or $\CPM(m,s,2n;r) \cong \CPM(2m,s,n;r')$, respectively, where $r'$ is $r$ mod $n$. 
\end{itemize}
\end{lemma}

\begin{proof}
Suppose first that $m$ is odd. Since $r^{ms} = \pm 1$, we have that $r^{2ms} = 1$, and so the graph $\CPM(2m,s,2n;r)$ is well defined. Let $\Psi \colon \CPM(2m,s,2n;r) \to \CPM(m,s,2n;r)$ be the mapping given by the rule
$$
	\laa i ; \mb{v}\raa\Psi = \laa i_{ms} ; \mb{v}\raa,
$$
where $i_{ms}$ is $i$ mod $ms$. To see that $\Psi$ is a mapping into $\CPM(m,s,2n;r)$ and that it is injective we only need to observe that if for some $0 \leq i \leq ms-1$ the vertices of the form $\laa i ; \mb{v}\raa$ and $\laa i+ms ; \mb{u}\raa$ are both in $\CPM(2m,s,2n;r)$, then for each $0 \leq \ell < s$ the components $v_\ell$ and $u_{\ell}$ are of different parity (as $m$ is odd), and so $\mb{v} \neq \mb{u}$ and $\laa i ; \mb{v}\raa\Psi = \laa i_{ms}; \mb{v}\raa$ and $\laa i + ms ; \mb{u}\raa\Psi = \laa i_{ms} ; \mb{u}\raa$ both belong to $\CPM(m,s,2n;r)$. That $\Psi$ preserves adjacency follows directly from $r^{ms} = \pm 1$, and so $\Psi$ is an isomorphism of graphs.

Suppose next that $n$ is odd and let $r'$ be $r$ mod $n$. It is clear that $r'^{ms} = \pm 1$, so that the graphs $\CPM(m,s,n;r')$ and $\CPM(2m,s,n;r')$ are well defined. Now, if $m$ is even then one can easily verify that the mapping $\Psi' \colon \CPM(m,s,2n;r) \to \CPM(m,s,n;r')$, given by the rule $\laa i ; \mb{v}\raa \Psi' = \laa i ; \mb{v'}\raa$, where $\mb{v'}$ is obtained from $\mb{v}$ by computing each of its components modulo $n$, is an isomorphism of graphs. If however $m$ is odd, then (i) and the first part of this paragraph imply $\CPM(m,s,2n;r) \cong \CPM(2m,s,2n;r) \cong \CPM(2m,s,n;r')$.
\end{proof}

We now record some natural automorphisms of the graphs $\CPMall(m,s,n;r)$. The proof of Proposition~\ref{pro:HATgroup} is rather straightforward but the obtained group $\Gall$ and the corresponding group $G$ acting on $\CPM(m,s,n;r)$ will play a crucial role in the sections to follow. When trying to verify the claims of Proposition~\ref{pro:HATgroup} the following alternative viewpoint regarding the $\CPM$ graphs may be of help. Namely, the graph $\CPMall(m,s,n;r)$ can alternatively be described as a regular cover of the ``doubled cycle'' of length $ms$ (the multigraph obtained from the cycle of length $ms$ by replacing each edge by a pair of parallel edges) with voltage group $\ZZ_{n}^s$ where the two parallel arcs from the vertex $i$ to the vertex $i+1$ are assigned voltages $\pm r^i \mb{e_\ell}$, where $\ell$ is $i$ mod $s$ (for an overview of basic concepts and results from the theory of graph covers via voltage assignments see for instance \cite{Mal98, MalNedSko00}). This explains where certain natural automorphisms of the $\CPM$ graphs come from.

\begin{proposition}
\label{pro:HATgroup}
Let $m, s, n$ be positive integers with $n \geq 3$ and let $r \in \ZZ_n^*$ be such that $r^{ms} = \pm 1$. Let $\Gmall = \CPMall(m,s,n;r)$ and $\G = \CPM(m,s,n;r)$. For each $0 \leq \ell \leq s-1$ let $\rhoall_\ell$ and $\tauall_\ell$ be the permutations of the vertex set of $\Gmall$, given by the rules
\begin{equation}
\label{eq:rho}
\laa i; \mb{v}\raa \rhoall_\ell = \laa i; \mb{v} + \mb{e_\ell}\raa,\ i \in \ZZ_{ms},\ \mb{v} \in \ZZ_n^s,
\end{equation}
\begin{equation}
\label{eq:tau}
\laa i; \mb{v}\raa \tauall_\ell = \laa i; (v_0, v_1, \ldots , v_{\ell - 1}, -v_\ell, v_{\ell + 1}, \ldots , v_{s-1})\raa,\ i \in \ZZ_{ms},\ \mb{v} \in \ZZ_n^s.
\end{equation}
Then $\rhoall_\ell$ and $\tauall_\ell$ are automorphisms of $\Gmall$ and for each $0 \leq \ell, \ell' \leq s-1$ we have that 
\begin{equation}
\label{eq:taurho}
\tauall_{\ell'} \rhoall_\ell \tauall_{\ell'} = \left\{\begin{array}{ccl}	\rhoall_\ell^{-1} & ; & \ell = \ell', \\ \rhoall_{\ell} & ; & \ell \neq \ell'.\end{array}\right.
\end{equation}
In addition, $\Rall = \la \rhoall_0, \rhoall_1, \ldots , \rhoall_{s-1} \ra \cong \ZZ_n^s$, $\Tall = \la \tauall_0, \tauall_1, \ldots , \tauall_{s-1} \ra \cong \ZZ_2^s$ and $\Tall \leq \Aut(\Gmall)_{(0; \mb{0})}$.\\
Similarly, the permutation of the vertex set of $\Gmall$, given by the rule
\begin{equation}
\label{eq:sigma}
\laa i; \mb{v}\raa\sigmaall = \laa i+1;(r v_{s-1}, r v_0, r v_1, \ldots , r v_{s-2})\raa, \ i \in \ZZ_{ms},\ \mb{v} \in \ZZ_n^s,
\end{equation}
is an automorphism of $\Gmall$. Moreover, 
$$
	\sigmaall^{-1}\rhoall_\ell \sigmaall = \rhoall_{\ell + 1}^r\quad \text{and}\quad \sigmaall^{-1} \tauall_\ell \sigmaall = \tauall_{\ell + 1},\quad  0 \leq \ell \leq s-1,
$$
where the index $\ell + 1$ is computed modulo $s$. In addition, the subgroup $\Gall = \la \sigmaall, \rhoall_0, \tauall_0 \ra \leq \Aut(\Gmall)$ acts half-arc-transitively on $\Gmall$ with vertex stabilizers isomorphic to $\ZZ_2^s$ and the sets $\Vall_i$ from \eqref{eq:V_i} are blocks of imprimitivity for the action of $\Gall$. Similarly, the restriction $G$ of the action of the setwise stabilizer of $V(\G)$ in $\Gall$ to $V(\G)$ is half-arc-transitive with vertex-stabilizers isomorphic to $\ZZ_2^s$ and the sets $V_i$ from \eqref{eq:V_i} are blocks of imprimitivity for the action of $G$.
\end{proposition}

\begin{proof}
The straightforward verifications that $\rhoall_\ell$ and $\tauall_\ell$ are indeed automorphisms of $\Gmall$ and that \eqref{eq:taurho} and the claims about $\Rall$ and $\Tall$ hold are left to the reader. To see that $\sigmaall$ is also an automorphism of $\Gmall$ let $\mb{v} \in \ZZ_n^s$, let $i \in \ZZ_{ms}$ and let $\ell$ be $i$ mod $s$. The two neighbors $\laa i+1; \mb{v} \pm r^{i}\mb{e_\ell}\raa$ of $\laa i; \mb{v}\raa$ in $\Vall_{i+1}$ are mapped by $\sigmaall$ to 
$$
	\laa i+1; \mb{v} \pm r^{i}\mb{e_\ell}\raa\sigmaall = \laa i+2; (r v_{s-1}, r v_0, r v_1, \ldots , r v_{\ell-1}, r (v_{\ell} \pm r^i), r v_{\ell+1}, \ldots, rv_{s-2})\raa,
$$
which are clearly neighbors of $\laa i ; \mb{v}\raa\sigmaall$ (recall that $r^{ms} = \pm 1$), and so $\sigmaall$ is indeed an automorphism of $\Gmall$. The remaining claims about $\sigmaall$ and $\Gall$ are now only a matter of a few simple calculations, which are left to the reader. 

To prove the final part of the proposition assume $n$ is even (so that $\G \neq \Gmall$). Observe that each $\tauall_i$ preserves $\G$, and so the claim that the vertex-stabilizers in $G$ are isomorphic to $\ZZ_2^s$ is clear. Since $\sigmaall\rhoall_0$ and all of the $\rhoall_i^2$ also preserve $\G$, $G$ acts vertex-transitively, and thus also half-arc-transitively on $\G$.
\end{proof}

\begin{remark}
The reader will observe that the orbits of the subgroup $\Rall$ from Proposition~\ref{pro:HATgroup} coincide with the sets $\Vall_i$ from \eqref{eq:V_i}.
Observe also that the nature of the half-arc-transitive action of the group $G$ from Proposition~\ref{pro:HATgroup} is such that if we orient the edge $\laa 0;\mb{0}\raa \laa 1;\mb{e_0}\raa$ from $\laa 0;\mb{0}\raa$ to $\laa 1;\mb{e_0}\raa$, then the corresponding $G$-induced orientation of the edges of $\G = \CPM(m,s,n;r)$ is such that each edge of $\G[V_i, V_{i+1}]$ is oriented from $V_i$ to $V_{i+1}$. Proposition~\ref{pro:HATgroup} thus implies that $G$ acts transitively on the set of $s$-arcs of the corresponding oriented graph. Moreover, it is now clear that the corresponding radius $\rad_G(\G)$ is $n$ or $n/2$, depending on whether $n$ is odd or even, respectively. Furthermore, $\G$ is tightly $G$-attached if and only if $s = 1$ and is loosely $G$-attached otherwise. Since the tetravalent graphs admitting a half-arc-transitive group of automorphisms relative to which they are tightly attached have already been extensively studied~\cite{Mar98, MarPra99, Spa08, Wil04}, we hereafter restrict to the graphs with $s \geq 2$.
\end{remark}

\section{More isomorphisms and additional symmetries}
\label{sec:symiso}

Our primary goal in this paper is to determine the automorphism group of each $\CPM$ graph and to classify the $2$-arc-transitive and the half-arc-transitive members of this family. In view of Proposition~\ref{pro:HATgroup} each $\CPM$ graph admits a half-arc-transitive group of automorphisms (the group $G$), and so it is either half-arc-transitive or arc-transitive. We thus need to determine the situations in which additional automorphisms (not from $G$), making the graph arc-transitive or even $2$-arc-transitive, exist.

To achieve this it clearly suffices to restrict to the connected component $\CPM(m,s,n;r)$ which is what we shall do most of the time in the sections to follow. Nevertheless, it often happens that it is easier to describe certain automorphism (or isomorphisms) for the whole $\CPMall(m,s,n;r)$, so we sometimes work with the whole graphs instead.

\subsection{Isomorphisms}
\label{subsec:iso}

We first record some more isomorphisms between the $\CPM$ graphs. As we shall see, these also enable us to find additional automorphisms of the $\CPM$ graphs in certain situations (see Proposition~\ref{pro:AT}).

\begin{lemma}
\label{le:iso}
Let $m, s, n$ be positive integers with $n \geq 3$ and let $r \in \ZZ_n^*$ be such that $r^{ms} = \pm 1$. Then $\CPM(m,s,n;r) = \CPM(m,s,n;-r)$ and $\CPM(m,s,n;r) \cong \CPM(m,s,n;r^{-1})$.
\end{lemma}

\begin{proof}
As mentioned in the above paragraph we prefer to work with the whole graphs $\CPMall$. That $\CPMall(m,s,n;r) = \CPMall(m,s,n;-r)$ is clear from the definition. It is also easy to see that the mapping $\Psi \colon \CPMall(m,s,n;r) \to \CPMall(m,s,n;r^{-1})$, defined by the rule
$$
	\laa i ; \mb{v}\raa\Psi = \laa -i ; (rv_{s-1}, rv_{s-2}, \ldots , rv_1, rv_0)\raa, \ i \in \ZZ_{ms},\ \mb{v} \in \ZZ_n^s,
$$
is an isomorphism of graphs. We leave details to the reader.
\end{proof}

\begin{proposition}
\label{pro:iso2}
Let $m, s, n$ be positive integers with $n \geq 3$ and let $r, r' \in \ZZ_n^*$ be such that $r^{ms} = \pm 1$ and $r'^{ms} = \pm 1$. If any of the following holds
\begin{itemize}\itemsep = 0pt
\item[(i)] $(rr')^s  = \pm 1$ or $(r^{-1}r')^s = \pm 1$;
\item[(ii)] $m$ is divisible by $4$ and either $2((rr')^s \pm 1) = 0$ or $2((r^{-1}r')^s \pm 1) = 0$,
\end{itemize}
then $\CPM(m,s,n;r) \cong \CPM(m,s,n;r')$.
\end{proposition}

\begin{proof}
We again describe the isomorphisms between the whole $\CPMall$ graphs. Suppose first that (i) holds. By Lemma~\ref{le:iso} we have that $\CPMall(m,s,n;r) \cong \CPMall(m,s,n;r^{-1})$, and so we can assume that $(r^{-1}r')^s = \pm 1$ holds. Let $q = r^{-1}r'$ and note that $q \in \ZZ_n^*$. Let $\Phi \colon \CPMall(m,s,n;r) \to \CPMall(m,s,n;r')$ be the mapping defined by the rule
$$
	\laa i ; \mb{v}\raa\Phi = \laa i ; (v_0, q v_1, q^2 v_2, \ldots , q^{s-1} v_{s-1})\raa.
$$ 
Since $q \in \ZZ_n^*$, the mapping $\Phi$ is a bijection. To see that it also preserves adjacency let $\laa i ; \mb{v}\raa$ be a vertex of $\CPMall(m,s,n;r)$ and let $\ell$ be $i$ mod $s$. The two neighbors $\laa i+1 ; \mb{v} \pm r^i \mb{e_\ell}\raa$ of $\laa i ; \mb{v}\raa$ in $\Vall_{i+1}$ are then mapped to 
$$
	\laa i+1 ; (v_0, q v_1, q^2 v_2, \ldots , q^{\ell-1} v_{\ell - 1}, q^\ell (v_\ell \pm r^i), q^{\ell + 1} v_{\ell + 1}, \ldots , q^{s-1}v_{s-1})\raa.
$$
For these to be neighbors of $\laa i ; \mb{v}\raa\Phi$ we require that for each $\delta \in \{-1,1\}$ there exists a $\delta' \in \{-1,1\}$ such that $q^\ell v_\ell + \delta' r'^i = q^\ell(v_\ell + \delta r^i)$, which is equivalent to $\delta' r'^i = \delta q^\ell r^i$. Multiplying by $r^{-i}$ we obtain $\delta' q^i = \delta q^\ell$. Since $q^s = \pm 1$ and $\ell$ is $i$ mod $s$, we have that $q^i = \pm q^\ell$, which thus proves that $\Phi$ is an isomorphism of graphs.

Suppose now that (i) does not hold but (ii) does. As above Lemma~\ref{le:iso} implies that we can assume $2((r^{-1}r')^s \pm 1) = 0$. Denote again $q = r^{-1}r'$. Since $q^s \neq \pm 1$ and $q \in \ZZ_n^*$, it follows that $n = 4\tilde{n}$ for some integer $\tilde{n}$ and that $q^s \pm 1 = 2\tilde{n}$. Let $\Phi' \colon \CPMall(m,s,4\tilde{n};r) \to \CPMall(m,s,4\tilde{n};r')$ be the mapping defined by the rule
$$
	\laa i ; \mb{v}\raa \Phi' = \laa i ; (v_0 + \delta_0, q v_1 + \delta_1, q^2 v_2 + \delta_2, \ldots , q^{s-1} v_{s-1} + \delta_{s-1})\raa,
$$ 
where letting $i_{4s} \in \{0,1,\ldots , 4s-1\}$ be the residue of dividing $i$ by $4s$ we set 
$$
	\delta_j = \left\{\begin{array}{ccc} 2\tilde{n} & ; & s \leq i_{4s} - j - 1 \leq 3s-1, \\ 0 & ; & \mathrm{otherwise}\end{array}\right.
$$
for each $0 \leq j < s$.
That $\Phi'$ is a bijection is clear. The verification that it preserves adjacency is somewhat tedious, but straightforward. We provide here some details and leave the rest of the verification to the reader. For the vertices of the form $\laa i ; \mb{v} \raa$ with $i \leq s$ all of the $\delta_j$ are zero, and so these vertices are mapped by $\Phi'$ in the same way they were mapped by $\Phi$ in the previous paragraph. The first edges to be checked are thus of the form $\laa s; \mb{v}\raa \laa s+1; \mb{v} \pm r^s\mb{e_0}\raa$. We have
$$
	\laa s+1 ; \mb{v} \pm r^s\mb{e_0} \raa \Phi' = \laa s+1 ; (v_0\pm r^s + 2\tilde{n}, qv_1, q^2 v_2, \ldots , q^{s-1}v_{s-1})\raa.
$$
As $q^s = 2\tilde{n} \pm 1$, we have that $r'^s = 2\tilde{n} \pm r^s$, and so each of $\laa s+1 ; \mb{v} \pm r^s\mb{e_0} \raa \Phi'$ is a neighbor of $\laa s ; \mb{v} \raa \Phi'$. To look at just one more example consider the edges of the form $\laa 3s+1; \mb{v}\raa \laa 3s+2; \mb{v} \pm r^{3s+1}\mb{e_{1}}\raa$. For $\laa 3s+1; \mb{v}\raa$ we have that $\delta_j = 2\tilde{n}$ for all $j > 0$ and $\delta_0 = 0$. Therefore,
$$
	\laa 3s+1 ; \mb{v} \raa \Phi' = \laa 3s+1 ; (v_0, qv_1+2\tilde{n}, q^2 v_2+2\tilde{n}, \ldots , q^{s-1}v_{s-1}+2\tilde{n})\raa.
$$
Similarly, 
$$
	\laa 3s+2 ; \mb{v} \pm r^{3s+1}\mb{e_1} \raa \Phi' = \laa 3s+2 ; (v_0, qv_1\pm qr^{3s+1}, q^2 v_2+2\tilde{n}, q^3 v_3+2\tilde{n}, \ldots , q^{s-1}v_{s-1}+2\tilde{n})\raa.
$$
Since $q^{2s} = 1$, it follows that $q^{3s} = q^s = 2\tilde{n} \pm 1$, and so $qv_1 + 2\tilde{n} \pm r'^{3s+1} = qv_1 \pm qr^{3s+1}$, showing that each of $\laa 3s+2 ; \mb{v} \pm r^{3s+1}\mb{e_1} \raa \Phi'$ is a neighbor of $\laa 3s+1 ; \mb{v} \raa \Phi'$.
\end{proof}

\subsection{Arc-transitivity}
\label{subsec:AT}

We now identify certain situations in which the full automorphism group of $\CPM(m,s,n;r)$ contains automorphisms that are not contained in the subgroup $G$ from Proposition~\ref{pro:HATgroup} and ensure that the graph is arc-transitive. As will be proved later (see Theorem~\ref{the:HAT}) these are in fact the only situations in which the graph $\CPM(m,s,n;r)$ is not half-arc-transitive. 

\begin{proposition}
\label{pro:AT}
Let $m, s, n$ be positive integers with $s \geq 2$, $n \geq 3$ and let $r \in \ZZ_n^*$ be such that $r^{ms} = \pm 1$. If $2(r^{2s} \pm 1) = 0$ then the graph $\G = \CPM(m,s,n;r)$ is arc-transitive. 
\end{proposition}

\begin{proof}
Let $\Gmall = \CPMall(m,s,n;r)$. In view of Proposition~\ref{pro:HATgroup} it suffices to find an automorphism of $\Gmall$, which fixes the set $\Vall_0$ setwise and interchanges each $\Vall_i$ with $\Vall_{ms-i}$, $1 \leq i \leq ms-1$, where the sets $\Vall_i$ are as in \eqref{eq:V_i}. We obtain such an automorphism from the isomorphism from the proof of Lemma~\ref{le:iso} and the proof of Proposition~\ref{pro:iso2}. To this end let $r' = r^{-1}$, let $\Gmall' = \CPM(m,s,n;r')$ and let $\Psi$ be the isomorphism from the proof of Lemma~\ref{le:iso}. Observe that $\Psi$ maps the set $\Vall_0$ of $\Gmall$ to the set $\Vall_0$ of $\Gmall'$ and for each $1 \leq i < ms$ maps the set $\Vall_i$ of $\Gmall$ to the set $\Vall_{ms-i}$ of $\Gmall'$.

Suppose first that $r^{2s} \pm 1 = 0$. Then $(r^{-1}r')^s = r^{-2s} = \pm 1$, and so the mapping $\Phi$ from the proof of Proposition~\ref{pro:iso2} is an isomorphism from $\Gmall$ to $\Gmall'$. Note that $\Phi$ maps each set $\Vall_i$ of $\Gmall$ to the set $\Vall_i$ of $\Gmall$. Therefore, the mapping $\Phi \Psi^{-1}$ is an automorphism of $\Gmall$ having the required properties. Similarly, if $r^{2s} \pm 1 \neq 0$ but $2(r^{2s} \pm 1) = 0$, then $n = 4\tilde{n}$ for some integer $\tilde{n}$ and $r^{4s} = 1$. It follows that $(r^{-1}r')^s \pm 1 = r^{2s} \pm 1 = 2\tilde{n}$, and so the mapping $\Phi'$ from the proof of Proposition~\ref{pro:iso2} is an isomorphism from $\Gmall$ to $\Gmall'$. We can now proceed as in the case when $r^{2s} \pm 1 = 0$.
\end{proof}

\begin{remark}
A straightforward calculation confirms that in the case that $r^{2s} \pm 1 = 0$ the automorphism $\etaall = \Phi\Psi^{-1}$ from the above proof maps according to the rule
\begin{equation}
\label{eq:eta}
\laa i ; \mb{v}\raa\etaall = \laa -i ; (r^{-2s+1}v_{s-1}, r^{-2s+3}v_{s-2}, \ldots , r^{-3}v_1, r^{-1}v_0)\raa, \ i \in \ZZ_{ms},\ \mb{v} \in \ZZ_n^s.
\end{equation}
Similarly, in the case that $r^{2s} \pm 1 \neq 0$ but $2(r^{2s} \pm 1) = 0$ (recall that this implies that $n = 4\tilde{n}$ for some integer $\tilde{n}$ and that $m$ is also divisible by $4$) the automorphism $\etaall' = \Phi'\Psi^{-1}$ maps according to the rule
\begin{equation}
\label{eq:eta'}
\laa i ; \mb{v}\raa \etaall' = \laa -i ; (r^{-2s+1}v_{s-1}+\delta_0, r^{-2s+3}v_{s-2}+\delta_1, \ldots ,r^{-1}v_0+\delta_{s-1})\raa, \ i \in \ZZ_{ms},\ \mb{v} \in \ZZ_n^s,
\end{equation}
where letting $i_{4s} \in \{0,1,\ldots , 4s-1\}$ be the residue of dividing $i$ by $4s$ we set
$$
	\delta_j = \left\{\begin{array}{ccc} 2\tilde{n} & ; & 2s \leq i_{4s} + j \leq 4s-1,\\ 0 & ; & \mathrm{otherwise}\end{array}\right.
$$
for all $0 \leq j \leq s-1$.
\end{remark}

\subsection{$2$-arc-transitivity}
\label{subsec:2AT}

Observe that in the case of the situation from Proposition~\ref{pro:AT} the group generated by $\etaall$ from \eqref{eq:eta} (or $\etaall'$ from \eqref{eq:eta'}) and the group $\Gall$ from Proposition~\ref{pro:HATgroup} is arc-transitive on $\CPMall(m,s,n;r)$ but the sets $\Vall_i$ from \eqref{eq:V_i} are still blocks of imprimitivity for its action. This implies that this group does not act $2$-arc-transitively on the graph. Nevertheless, we now show that there do exist $2$-arc-transitive $\CPM$ graphs. 

\begin{lemma}
\label{le:aux2AT}
Let $k$ be an odd integer and let $q \in \ZZ_k^*$. Then the mapping $\beta \colon \ZZ_k^3 \to \ZZ_k^3$, given by the rule
$$
	(x,y,z)\beta = (y-qz, 2^{-1}(x+y+qz), 2^{-1}(z-q^{-1}x+q^{-1}y)),
$$
where $2^{-1} \in \ZZ_k^*$ denotes the multiplicative inverse of $2$ in $\ZZ_k$, is an involutory automorphism of the additive group $\ZZ_k^3$.
\end{lemma}

\begin{proof}
A straightforward computation shows that $\beta^2$ fixes each element of $\ZZ_k^3$, thus proving that it is an involutory permutation. That it is also additive is clear from the definition.
\end{proof}

\begin{proposition}
\label{pro:2AT}
Let $n \geq 3$ be an odd integer. Then the graph $\CPM(n,2,n;1)$ is $2$-arc-transitive. Moreover, for any integer $m \geq 2$ with $m \equiv 2 \pmod{4}$ and any $r \in \ZZ_{2m}^*$ with $r^2+1 = m$ the graph $\CPM(m,2,2m;r)$ is $2$-arc-transitive.
\end{proposition}

\begin{proof}
In each of the two cases $r^{2s} = r^4 = 1$, and so Proposition~\ref{pro:AT} implies that these graphs are all arc-transitive. In view of the action of the group $\Tall$ from Proposition~\ref{pro:HATgroup} it thus suffices to prove that there exists an automorphism fixing $\laa 0; (0,0) \raa$ and $\laa 1 ; (1,0)\raa$ while moving $\laa 1 ; (-1,0) \raa$.\smallskip

\noindent
{\sc Case 1}: $\Gmall = \CPMall(n,2,n;1)$ with $n \geq 3$ odd.\\
To define an appropriate automorphism of $\Gmall$ we first change the viewpoint regarding the vertices of $\Gmall$. Note that, since $n$ is odd, $\ZZ_{2n} \cong \ZZ_2 \times \ZZ_n$, and so we can regard the vertex-set of $\Gmall$ as being $\ZZ_2 \times \ZZ_n^3$ (we write the elements of this set as pairs $\laa e ; (x,y,z)\ra$, $e \in \ZZ_2$ and $x,y,z \in \ZZ_n$) in the sense that the vertex $\laa i ; (y,z) \raa$ of $\Gmall$ corresponds to $\laa i_2; (i_n, y, z)\raa$, where $i_2$ is $i$ mod $2$ and $i_n$ is $i$ mod $n$. Adjacency in $\Gmall$ thus translates as follows. For any $\mb{v} \in \ZZ_n^3$ we have that 
$$
\laa 0; \mb{v}\raa \sim \laa 0; \mb{v}\raa + \laa 1; \mb{f}\raa\ \text{and}\ \laa 1; \mb{v}\raa \sim \laa 1; \mb{v}\raa - \laa 1; \mb{f}\raa,
$$
where $\mb{f} \in \{(1,1,0), (1,-1,0), (-1,0,1), (-1,0,-1)\}$. 
Let now $\nuall$ be the permutation of the vertex-set $\ZZ_2 \times \ZZ_n^3$ given by the rule
\begin{equation}
\label{eq:nu}
	\la e; \mb{v} \raa \nuall = \la e; \mb{v}\beta \raa,
\end{equation}
where $\beta$ is as in Lemma~\ref{le:aux2AT} for $q = 1$. By Lemma~\ref{le:aux2AT} the mapping $\nuall$ is indeed bijective and is in fact an involutory automorphism of the additive group $\ZZ_2 \times \ZZ_n^3$. Since $\beta$ fixes setwise the above set of the four possibilities for $\mb{f}$ (it fixes the first and the third and interchanges the remaining two), it thus follows that $\nuall$ is an automorphism of $\Gmall$. Since it fixes the vertices $\laa 0 ; (0,0,0) \raa$ (corresponding to $\laa 0 ; (0,0) \raa$) and $\laa 1; (1,1,0) \raa$ (corresponding to $\laa 1; (1,0) \raa$) and maps $\laa 1; (1,-1,0) \raa$ (corresponding to $\laa 1 ; (-1,0) \raa$) to $\laa 1 ; (-1,0,-1)\raa$ (corresponding to $\laa -1 ; (0,-1)\raa$), this proves that $\Gmall$ is indeed $2$-arc-transitive.
\smallskip

\noindent
{\sc Case 2}: $\Gmall = \CPMall(m,2,2m;r)$ with $m \equiv 2 \pmod 4$ and $r^2 + 1 = m$.\\
If $m = 2$, then $\CPM(2,2,4;1)$ is the skeleton of the $4$-cube which is known to be $2$-arc-transitive, and so we can assume that $m > 2$. Note that $r^3 = m - r$ and $r^4 = 1$. Write $m = 2k$ and note that $k$ is odd. The vertex-set of $\Gmall$ can then be regarded as $\ZZ_{4k}^3$ in which case adjacency can be described in such a way that for each $\mb{v} = (x,y,z) \in \ZZ_{4k}^3$ we have that $(x,y,z) \sim (x,y,z) + \mb{f}$, where 
\begin{equation}
\label{eq:adj2AT}
	\mb{f} = \left\{\begin{array}{ccl}
			(1, \pm 1, 0) & \mathrm{if} & x \equiv 0 \pmod{4}\\
			(1, 0, \pm r) & \mathrm{if} & x \equiv 1 \pmod{4}\\
			(1, \pm r^2, 0) & \mathrm{if} & x \equiv 2 \pmod{4}\\
			(1, 0, \pm r^3) & \mathrm{if} & x \equiv 3 \pmod{4}.\end{array}\right.
\end{equation}
It thus follows that reducing each of the three components of the vertices modulo $4$ yields a graph epimorphism $\wp_4$ from $\Gmall$ to $\CPMall(2,2,4;1)$ (in fact, it is a covering projection). One possible way to complete the proof is thus to show that the entire automorphism group of $\CPMall(2,2,4;1)$ lifts along this covering projection. To make the paper independent of the theory of lifting automorphisms along covering projections we provide a direct construction of an appropriate automorphism. 

By Lemma~\ref{le:iso} we can assume $r \equiv 3 \pmod {4}$. Let $q$ be $r$ mod $k$. Then
\begin{equation}
\label{eq:tab_4q}
 \begin{array}{|c|c|}
 	\hline
 	r \equiv -1 \pmod{4} & r \equiv q \pmod{k} \\ \hline
 	r^2 \equiv 1 \pmod{4} & r^2 \equiv -1 \pmod{k} \\ \hline
 	r^3 \equiv -1 \pmod{4} & r^3 \equiv -q \pmod{k} \\ \hline 
 \end{array}
\end{equation}
Let $\nu_4$ be the automorphism of $\CPM(2,2,4;1)$ corresponding to the reflection of the toroidal embedding of $\CPM(2,2,4;1)$ given in Figure~\ref{fig:4-cube} with respect to the axis through the cycle with vertices $(2,3,1)$, $(3,2,1)$, $(0,2,2)$ and $(1,3,2)$, where we denote the vertices in the above agreed manner.
\begin{figure}[h]
\begin{center}
	\includegraphics[scale=1.2]{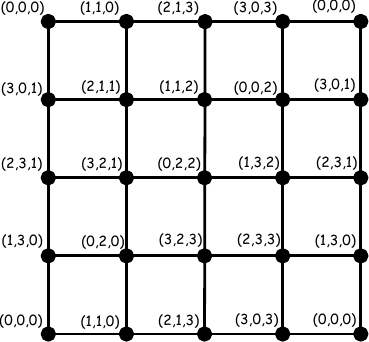}
	\caption{A toroidal embedding of the graph $\CPM(2,2,4;1)$.}
	\label{fig:4-cube}
\end{center}
\end{figure}
Let $\wp_k$ be the epimorphism of the additive group $\ZZ_{4k}^3$ onto $\ZZ_k^3$, corresponding to the reduction of all three components modulo $k$. Finally, let $\nu_k$ be the automorphism of the additive group $\ZZ_k^3$ corresponding to $\beta$ from Lemma~\ref{le:aux2AT} (with the above defined $q$). Now, let $\mb{v} = (x,y,z)$ be a vertex of $\G = \CPM(m,2,2m;r)$. By the Chinese remainder theorem there exists a unique $\mb{v}' \in \ZZ_{4k}^3$, such that $\mb{v}' \wp_4 = \mb{v}\wp_4 \nu_4$ and $\mb{v}' \wp_k = \mb{v}\wp_k \nu_k$. Define $\nu$ by setting $\mb{v} \nu = \mb{v}'$. Because $\nu_4$ and $\nu_k$ are involutions, $\nu$ itself is an involution, and thus a permutation of the vertex-set of $\G$. We claim that $\nu$ in fact preserves adjacency and is thus an automorphism of $\G$. 

To verify this we need to see that for each vertex $\mb{v}$ of $\G$ and each of its two neighbors of the form $\mb{v} + \mb{f}$, where $\mb{f}$ is as in (\ref{eq:adj2AT}), the vertices $\mb{v}\nu$ and $(\mb{v} + \mb{f})\nu$ are also adjacent. Now, $(\mb{v} + \mb{f})\nu\wp_4 = (\mb{v} + \mb{f})\wp_4\nu_4$ is adjacent to $\mb{v}\wp_4\nu_4$ (since $\wp_4\nu_4$ is a graph homomorphism), and is thus of the form $\mb{v}\wp_4\nu_4 + \mb{g} = \mb{v}\nu\wp_4 + \mb{g}$ for an appropriate $\mb{g} \in \{(1,\pm 1,0), (1,0,\pm 1), (-1,\pm 1, 0), (-1, 0, \pm 1)\}$. On the other hand, as $\wp_k$ and $\nu_k$ are both additive, we have that $(\mb{v} + \mb{f})\nu\wp_k = (\mb{v} + \mb{f})\wp_k\nu_k = \mb{v}\wp_k\nu_k + \mb{f}\wp_k\nu_k = \mb{v}\nu\wp_k + \mb{f}\wp_k\nu_k$. Therefore, $(\mb{v} + \mb{f})\nu = \mb{v}\nu + \mb{f}'$, where $\mb{f'} \in \ZZ_{4k}^3$ is the unique element such that $\mb{f'}\wp_4 = \mb{g}$ and $\mb{f'}\wp_k = \mb{f}\wp_k\nu_k$. 
This makes the verification that $\nu$ is an automorphism of $\G$ fairly straightforward. For each of the $16$ possibilities for $\mb{v}\wp_4$ and each of the two corresponding $\mb{f}$ from (\ref{eq:adj2AT}) we first calculate the above $\mb{g}$ and $\mb{f}\wp_k\nu_k$ and then use (\ref{eq:tab_4q}) to determine $\mb{f}'$. Based on the first coordinates of $\mb{v}\nu\wp_4$ and $(\mb{v}+\mb{f})\nu\wp_4$ we then confirm that $\mb{v}\nu$ and $(\mb{v} + \mb{f})\nu = \mb{v}\nu + \mb{f}'$ are indeed adjacent. To expedite the verifications we provide the following table in which for each of the four possibilities for $x\ \mathrm{mod}\ 4$, where $\mb{v} = (x,y,z)$, and each of the two corresponding possibilities for $\mb{f}$ from (\ref{eq:adj2AT}) the values of $\mb{f}\wp_4$, $\mb{f}\wp_k$ and $\mb{f}\wp_k\nu_k$ are given.
\begin{center}
\begin{tabular} { |c|c|c|c|c|}
\hline
$x$ mod 4&$\mb{f}$& $\mb{f}\wp_4$ & $\mb{f}\wp_k$ & $\mb{f}\wp_k\nu_k$\\
\hline
\multirow{2}{*}{0} &  $ (1,1,0) $   &  $ (1,1,0) $ &  $ (1,1,0) $ &  $ (1,1,0) $  \\
                            &  $ (1,-1,0) $  &  $ (1,3,0) $  &  $ (1,-1,0) $  &  $ (-1,0,q) $\\
\hline
\multirow{2}{*}{1} &  $ (1,0,r) $  &  $ (1,0,3) $ &  $ (1,0,q) $ &  $ (1,0,q) $  \\
                            &  $ (1,0,-r) $  &  $ (1,0,1) $  &  $ (1,0,-q) $  &  $ (-1,1,0) $ \\
\hline
\multirow{2}{*}{2} &  $ (1,r^2,0) $   &  $ (1,1,0) $ &  $ (1,-1,0) $ &  $ (-1,0,q) $  \\
                            &  $ (1,-r^2,0) $  &  $ (1,3,0) $  &  $ (1,1,0) $  &  $ (1,1,0) $\\
\hline
\multirow{2}{*}{3} &  $ (1,0,r^3) $  &  $ (1,0,3) $ &  $ (1,0,-q) $ &  $ (-1,1,0) $  \\
                            &  $ (1,0,-r^3) $  &  $ (1,0,1) $  &  $ (1,0,q) $  &  $ (1,0,q) $ \\
\hline
\end{tabular}
\medskip
\end{center} 

We provide two of the $32$ cases that need to be verified and leave the remaining ones to the reader. Suppose first that $\mb{v}$ is a vertex of $\G$ such that $\mb{v}\wp_4 = (1,3,0)$ and let us consider the possibility $\mb{f} = (1,0,r)$. Then $\mb{v}\nu\wp_4 = \mb{v}\wp_4\nu_4 = (3,0,1)$, whose first coordinate is $3$. Similarly, $(\mb{v} + \mb{f})\nu\wp_4 = (\mb{v}\wp_4 + \mb{f}\wp_4)\nu_4 = (2,3,3)\nu_4 = (0,0,2)$, whose first coordinate is $0$. As $(0,0,2) = (3,0,1) + (1,0,1)$, we thus have $\mb{g} = (1,0,1)$. Since $\mb{f}\wp_k\nu_k = (1,0,q)$, (\ref{eq:tab_4q}) thus implies that $\mb{f}' = (1,0,-r^3)$, which by (\ref{eq:adj2AT}) confirms that $\mb{v}\nu$ is indeed adjacent to $(\mb{v} + \mb{f})\nu$. 

Similarly, if $\mb{v}\wp_4 = (2,3,1)$ and $\mb{f} = (1,r^2,0)$, then $\mb{v}\nu\wp_4 = (2,3,1)$ and $(\mb{v} + \mb{f})\nu\wp_4 = (3,0,1)\nu_4 = (1,3,0)$, and so $\mb{g} = (-1,0,-1)$. Since $\mb{f}\wp_k\nu_k = (-1,0,q)$, (\ref{eq:tab_4q}) implies that $\mb{f}' = (-1,0,r)$. Since the first coordinates of $\mb{v}\nu\wp_4$ and $(\mb{v} + \mb{f})\nu\wp_4$ are $2$ and $1$, respectively, (\ref{eq:adj2AT}) confirms that $\mb{v}\nu$ is indeed adjacent to $(\mb{v} + \mb{f})\nu$.

It is now also easy to verify that $\nu$ fixes both $(0,0,0)$ and $(1,1,0)$ and maps $(1,-1,0)$ to $(-1,0,-r^3)$, which thus shows that $\G$ and $\Gmall$ are both $2$-arc-transitive.
\end{proof}

One of the main goals of the next three sections is to show that the examples from Proposition~\ref{pro:2AT} are in fact the only $2$-arc-transitive $\CPM$ graphs.

\section{Short cycles and $2$-paths}
\label{sec:cycles2paths}

In the previous two sections our main goal was to prove that each $\CPM$ graph admits enough automorphisms to make it vertex- and edge-transitive (Proposition~\ref{pro:HATgroup}) and that in certain cases additional automorphisms making it arc-transitive or even $2$-arc-transitive exist (Proposition~\ref{pro:AT} and Proposition~\ref{pro:2AT}). From now on our main goal will be to obtain results that limit the degree of symmetry of $\CPM$ graphs and eventually lead to determination of the whole automorphism group of these graphs (Theorem~\ref{the:Aut2AT}, Theorem~\ref{the:HAT} and Corollary~\ref{cor:Aut_AT}).

Throughout this section we let $m, s, n$ be positive integers with $s \geq 2$ and $n \geq 3$, and we let $r \in \ZZ_n^*$ be such that $r^{ms} = \pm 1$. Furthermore, we let $\G = \CPM(m,s,n;r)$ and we let $G$ be the half-arc-transitive subgroup of $\Aut(\G)$ from Proposition~\ref{pro:HATgroup}.

Our method for the investigation of automorphisms of $\G$ is based on the ideas first used in~\cite{Mar98} and then successfully adapted for other similar situations (see for instance~\cite{RamSpa17, Spa08, Spa09}). The key idea is to investigate the possible orbits of $2$-paths of $\G$ under the action of $\Aut(\G)$ (or its subgroups) and their interplay with short cycles of $\G$. The obtained information about the local structure of the graph can then be used to determine the situations in which ``unexpected'' automorphisms of the graph in question may exist.
\medskip

We first introduce the terminology pertaining to $2$-paths of $\G$. We point out that for us a $2$-path does not have a direction (even though we will be giving them as ordered triples of vertices), that is, we consider a {\em $2$-path} simply as a corresponding subgraph of $\G$. If we want to speak of $2$-paths with orientation, we will refer to them as {\em $2$-arcs}. The key observation concerning the set of all $2$-paths of $\G$ is that, provided that $ms > 2$, we can partition it into two subsets depending on whether the two endvertices of the $2$-path both belong to the same set $V_i$ from \eqref{eq:V_i} or not. In accordance with the terminology from~\cite{Mar98} the $2$-paths of the first kind (where its endvertices belong to the same set $V_i$) will be called {\em anchors}, while the $2$-paths of the second kind will be called {\em non-anchors}. By Proposition~\ref{pro:HATgroup} the sets $V_i$ are blocks of imprimitivity for the action of the group $G$, and so the elements of $G$ map anchors to anchors and non-anchors to non-anchors. Moreover, since $s \geq 2$, the remark following Proposition~\ref{pro:HATgroup} implies that the set of all non-anchors of $\G$ is a $G$-orbit. On the other hand, the set of all anchors is composed of two $G$-orbits, one containing all anchors with mid-vertex in some $V_i$ and the remaining two vertices in $V_{i+1}$ (such anchors will be called {\em negative anchors}), and the other containing all anchors with mid-vertex in some $V_i$ and the remaining two vertices in $V_{i-1}$ (such anchors will be called {\em positive anchors}). On Figure~\ref{fig:generic} the non-anchor $(\laa 0;(1,2)\raa, \laa 1;(0,2)\raa, \laa 2;(0,4)\raa )$, the negative anchor $(\laa 1;(1,3)\raa, \laa 0;(2,3)\raa, \laa 1;(3,3)\raa)$ and the positive anchor $(\laa 0;(0,1)\raa, \laa 1;(1,1)\raa, \laa 0;(2,1)\raa )$ in the part of the graph $\CPM(3,2,5;2)$ induced on $V_0 \cup V_1 \cup V_2$ are indicated.

Using the above distinction between anchors and non-anchors of $\G$ a sequence of length $d$ can be assigned to each walk $W$ of length $d+1$ in $\G$, where we assign a symbol $a$ or $n$ to each internal vertex of $W$, depending on whether the corresponding $2$-subpath of $W$ with this mid-vertex is an anchor or a non-anchor, respectively. (Even though we use the same symbol $n$ as one of the defining parameters of the graphs $\CPM(m,s,n;r)$ as well as to represent a non-anchor, this should cause no confusion.) We call the obtained sequence the {\em code} of $W$.
\medskip

Of course, if $ms = 2$, then we cannot distinguish between anchors and non-anchors of $\G$. We therefore first consider this exceptional situation.

\begin{lemma}
\label{le:ms=2}
Let $n \geq 3$ be an integer and let $r \in \ZZ_n^*$ be such that $r^2 = \pm 1$. Then the graph $\G = \CPM(1,2,n;r)$ is arc-transitive. If $n = 4$, then $\G$ is isomorphic to the $4$-cube which is $2$-arc-transitive and the vertex-stabilizers in $\Aut(\G)$ are of order $24$. Otherwise $\G$ is not $2$-arc-transitive and the vertex-stabilizers in $\Aut(\G)$ are of order $8$.
\end{lemma}

\begin{proof}
By Proposition~\ref{pro:iso2} we can assume that $r = 1$. The claim about $\CPM(1,2,4;1) \cong \CPM(2,2,4;1)$ (see Lemma~\ref{le:discon}) follows from Proposition~\ref{pro:2AT}. For the rest of the proof we thus assume $n \neq 4$. Arc-transitivity of $\G$ follows from Proposition~\ref{pro:AT}. Since Proposition~\ref{pro:HATgroup} implies that the pointwise stabilizer of an arc is of order at least $2$, it thus suffices to prove that $\G$ is not $2$-arc-transitive and that the only automorphism of $\G$ fixing the vertex $\lea 0; (0,0)\ria$ and all of its neighbors is the identity. 

Consider the arc $P = (\lea 0;(0,0)\ria,\lea 1;(1,0)\ria)$. It is easy to verify that there is a unique $4$-cycle of $\G$ through each of $P\cdot (\lea 0;(1,1)\ria)$ and $P \cdot (\lea 0;(1,-1)\ria )$ (with $\lea 1;(0,1)\ria$ and $\lea 1;(0,-1)\ria$, respectively, being the corresponding remaining vertex of the $4$-cycle), while there is no $4$-cycle through $P \cdot (\lea 0;(2,0)\ria )$. This proves that $\G$ is not $2$-arc-transitive. Moreover, it shows that for each arc $(x,y)$ of $\G$ there is a unique neighbor $z_1$ of $y$, different from $x$, such that there is no $4$-cycle of $\G$ through $(x,y,z_1)$, while for each of the remaining two neighbors $z_i$ of $y$, $i \in \{2,3\}$, there is a unique neighbor $w_i$ of $x$, different from $y$, such that $(w_i,x,y,z_i)$ is a $4$-cycle. This shows that each automorphism $\alpha$ of $\G$ fixing $\lea 0; (0,0) \ria$ and all of its neighbors must also fix all neighbors of each of the neighbors of $\lea 0; (0,0) \ria$. Since $\G$ is connected, this shows that $\alpha$ is the identity.
\end{proof}

\begin{remark}
We point out that the graphs $\CPM(1,2,n;r)$ from Lemma~\ref{le:ms=2} all have the property that each of their edges lies on at least two $4$-cycles, and so by \cite[Theorem~3.3]{PotWil07} these graphs are all skeletons of edge-transitive toroidal maps of type $\{4,4\}$. More precisely, they are of type $\{4,4\}_{n,0}$ if $n$ is even and of type $\{4,4\}_{n,n}$ if $n$ is odd.
\end{remark}

Like the examples with $ms = 2$, the $\CPM$ graphs with the smallest possible even $n$, namely $n = 4$, also prove to be somewhat exceptional, which is why we will analyze them separately. In fact, as we show in the following proposition, all the corresponding $\CPM$ graphs are the well-known Praeger-Xu graphs, which thus also confirms a conjecture from~\cite[Section~15]{PotWil??}. To be able to state the corresponding result we recall the definition of these graphs and the result determining their automorphism group (but see~\cite{GarPra94, JajPotWil19, PraXu89} for more details). For integers $t \geq 3$ and $s \geq 1$ the Praeger-Xu graph $\mathrm{PX}(t,s)$ is the tetravalent graph whose vertex-set consists of all pairs $(i, \mb{v})$, where $i \in \ZZ_t$ and $\mb{v} \in \ZZ_2^s$, and in which each $(i, (v_0,v_1,\ldots , v_{s-1}))$ is adjacent to each of $(i+1, (v_1, v_2, \ldots , v_{s-1}, 0))$ and $(i+1, (v_1, v_2, \ldots , v_{s-1}, 1))$. It was shown in~\cite{PraXu89} that the graph $\mathrm{PX}(t,s)$ is arc-transitive if and only if $t > s$, and that in this case the vertex-stabilizers in the full automorphism group of the graph are of order $2^{t+1-s}$, unless $t = 4$ and $s \leq 3$. 

Recall that $\CPM(m,s,4;1) = \CPM(m,s,4;-1)$ and that by Lemma~\ref{le:discon} we have that $\CPM(m,s,4;r) \cong \CPM(2m,s,4;r)$ whenever $m$ is odd. We thus lose no generality by assuming $r = 1$ and that $m$ is even in the following proposition.

\begin{proposition}
\label{pro:n=4}
Let $m, s$ be positive integers with $m$ even and $s \geq 2$. Then the graph $\G = \CPM(m,s,4;1)$ is isomorphic to the Praeger-Xu graph $\mathrm{PX}(ms,s)$. If $m = s = 2$, then $\G$ is isomorphic to the $4$-cube which is $2$-arc-transitive and the vertex-stabilizers in $\Aut(\G)$ are of order $24$. Otherwise $\G$ is not $2$-arc-transitive and the vertex-stabilizers in $\Aut(\G)$ are of order $2^{s(m-1)+1}$.
\end{proposition}

\begin{proof}
Recall that, since $m$ is even, $\G$ is of order $ms2^s$ and for each $\lea i; \mb{v} \ria, \lea i ; \mb{v}' \ria \in V_i$, where $V_i$ is as in \eqref{eq:V_i}, the components $v_j, v'_j$ are of the same parity for each $0 \leq j \leq s-1$. 

To each $a \in \ZZ_4$ we assign the element $a^* \in \ZZ_2$ in such a way that $2^* = 3^* = 1$ and $0^* = 1^* = 0$. Consider now the map $\Psi \colon \G \to \mathrm{PX}(ms,s)$, defined by the rule
$$
	\lea i;\mb{v}\ria \Psi = (i, (v_\ell^*, v_{\ell + 1}^*, \ldots, v_{s-1}^*, v_0^*, v_1^*, \ldots , v_{\ell-1}^*)),
$$
where $\ell$ is $i$ mod $s$. The comment from the beginning of the proof ensures that the map $\Psi$ is bijective. To see that it also preserves adjacency simply observe that the two neighbors $\lea i+1;\mb{v} \pm \mb{e_\ell}\ria$ of $\lea i;\mb{v}\ria$ in $V_{i+1}$ are mapped by $\Psi$ to 
$$
	(i+1, (v_{\ell + 1}^*, \ldots, v_{s-1}^*, v_0^*, \ldots , v_{\ell-1}^*, v_{\ell}^*))\quad \text{and}\quad 
	(i+1, (v_{\ell + 1}^*, \ldots, v_{s-1}^*, v_0^*, \ldots , v_{\ell-1}^*, v_{\ell}^*+1)).
$$
The remaining claims of the proposition follow from~\cite[Theorem~2.13]{PraXu89}.
\end{proof}

In view of Lemma~\ref{le:ms=2} and Proposition~\ref{pro:n=4} we can now restrict to the $\CPM$ graphs with $ms \geq 3$ and $n \neq 4$. Note that this implies that we have at least three sets $V_i$ in $\G$, and so we have anchors as well as non-anchors in $\G$. For ease of reference in the rest of the paper we record our assumptions on the parameters $m$, $s$, $n$ and $r$.

\begin{assumption}
\label{as:1}
We let $m, s, n$ be positive integers where $n \geq 3$, $n \neq 4$, $s \geq 2$ and $ms \geq 3$, and we let $r \in \ZZ_n^*$ be such that $r^{ms} = \pm 1$.
\end{assumption}

The following simple but useful result shows why $2$-paths are important in the investigation of symmetry properties of the $\CPM$ graphs.

\begin{proposition}
\label{pro:blocks}
Let $m, s, n$ and $r$ be as in Assumption~\ref{as:1} and let $\G = \CPM(m,s,n;r)$. Then the sets $V_i$ from \eqref{eq:V_i} are blocks of imprimitivity for the action of $\Aut(\G)$ on $\G$ if and only if the set of non-anchors of $\G$ is an $\Aut(\G)$-orbit, which happens if and only if $\G$ is not $2$-arc-transitive. In addition, if $\G$ is $2$-arc-transitive, then for each vertex $x$ of $\G$ the restriction of the action of the stabilizer $(\Aut(\G))_x$ to the set of four neighbors of $x$ is permutation isomorphic to $S_4$.
\end{proposition}

\begin{proof}
Recall that the two end-vertices of an anchor of $\G$ belong to the same set $V_i$, while for the non-anchors this does not hold. Therefore, if the sets $V_i$ are blocks of imprimitivity for $\Aut(\G)$, Proposition~\ref{pro:HATgroup} implies that the set of all non-anchors of $\G$ is an $\Aut(\G)$-orbit.

Conversely, suppose that the set of all non-anchors of $\G$ is an $\Aut(\G)$-orbit, let $i \in \ZZ_{ms}$ and consider $\laa i ; \mb{v}\raa, \laa i ; \mb{v'}\raa \in V_i$. Since $\G$ is connected, there exists a walk $W$ from $\laa i ; \mb{v}\raa$ to $\laa i ; \mb{v'}\raa$. Note that for each internal vertex $\laa i'' ; \mb{v''}\raa$ of $W$ the parameter $i''$ is completely determined by the information on whether the first vertex on $W$ after $\laa i ; \mb{v}\raa$ is in $V_{i-1}$ or in $V_{i+1}$, and by the code of the subwalk of $W$ from $\laa i;\mb{v}\raa$ to $\laa i'', \mb{v''}\raa$. Since any element of $\Aut(\G)$ maps anchors to anchors and non-anchors to non-anchors, it preserves codes of paths, and so it maps the walk $W$ to a walk whose endvertices belong to the same set $V_{i'}$. This shows that the sets $V_i$ are blocks of imprimitivity for the action of $\Aut(\G)$.

To prove the second equivalence observe that if $\G$ is $2$-arc-transitive, the set of non-anchors of $\G$ is clearly not an $\Aut(\G)$-orbit. For the converse, assume the set of non-anchors is not an $\Aut(\G)$-orbit and let us show that in this case $\G$ is $2$-arc-transitive. Consider the six $2$-paths with a given mid-vertex. Four of them are non-anchors (and are thus all in the same $\Aut(\G)$-orbit by the remark following Proposition~\ref{pro:HATgroup}), one of them is a positive anchor and the remaining one a negative anchor. An automorphism that maps a given non-anchor with mid-vertex $x$ to an anchor with mid-vertex $y$ has to map the complementary non-anchor with mid-vertex $x$ to the complementary anchor with mid-vertex $y$. This shows that $\Aut(\G)$ has just one orbit on the set of all $2$-paths of $\G$. Since the group $G$ from Proposition~\ref{pro:HATgroup} provides automorphisms reversing the $2$-arc corresponding to a given anchor, this shows that the graph $\G$ is in fact $2$-arc-transitive. 

The last part follows from the fact that the only two $2$-transitive permutation groups of degree $4$ are the alternating group $A_4$ and the symmetric group $S_4$. Namely, since $G$ contains automorphisms fixing a given vertex $x$ and interchanging two of its neighbors while fixing the remaining two, the restriction of $(\Aut(\G))_x$ to the set of four neighbors of $x$ cannot be permutation isomorphic to $A_4$. Therefore, if $\G$ is $2$-arc-transitive, this restriction is isomorphic to $S_4$. 
\end{proof}

Suppose $\G = \CPM(m,s,n;r)$ is $2$-arc-transitive. Then all $2$-paths of $\G$ are in the same $\Aut(\G)$-orbit, and so for any $k \geq 3$ the number of $k$-cycles through any anchor has to be equal to the number of $k$-cycles through any non-anchor. We make use of this observation by investigating short cycles of $\G$ and count the number of such cycles through an anchor and the number of such cycles through a non-anchor. We first show that each $\CPM$ graph with $ms \geq 3$ has certain $8$-cycles. Indeed,   
\begin{equation}
\label{eq:generic}
	(\laa 0 ; \mb{0}\raa, \laa 1 ; \mb{e_0}\raa, \laa 2 ; \mb{e_0} + r\mb{e_1}\raa, \laa 1 ; \mb{e_0} + 2r\mb{e_1}\raa, \laa 0 ; 2r\mb{e_1}\raa, \laa 1 ; -\mb{e_0} + 2r\mb{e_1}\raa, \laa 2 ; -\mb{e_0} + r\mb{e_1}\raa, \laa 1 ; -\mb{e_0}\raa )
\end{equation}
is an $8$-cycle of $\G$ (see Figure~\ref{fig:generic} where this cycle is depicted in a section of $\CPM(3,2,5;2)$). An $8$-cycle of $\G$ is said to be {\em generic} if it is in the $G$-orbit of the $8$-cycle from~\eqref{eq:generic}, where $G$ is as in Proposition~\ref{pro:HATgroup}.
\begin{figure}[h]
\begin{center}
	\includegraphics[scale=0.4]{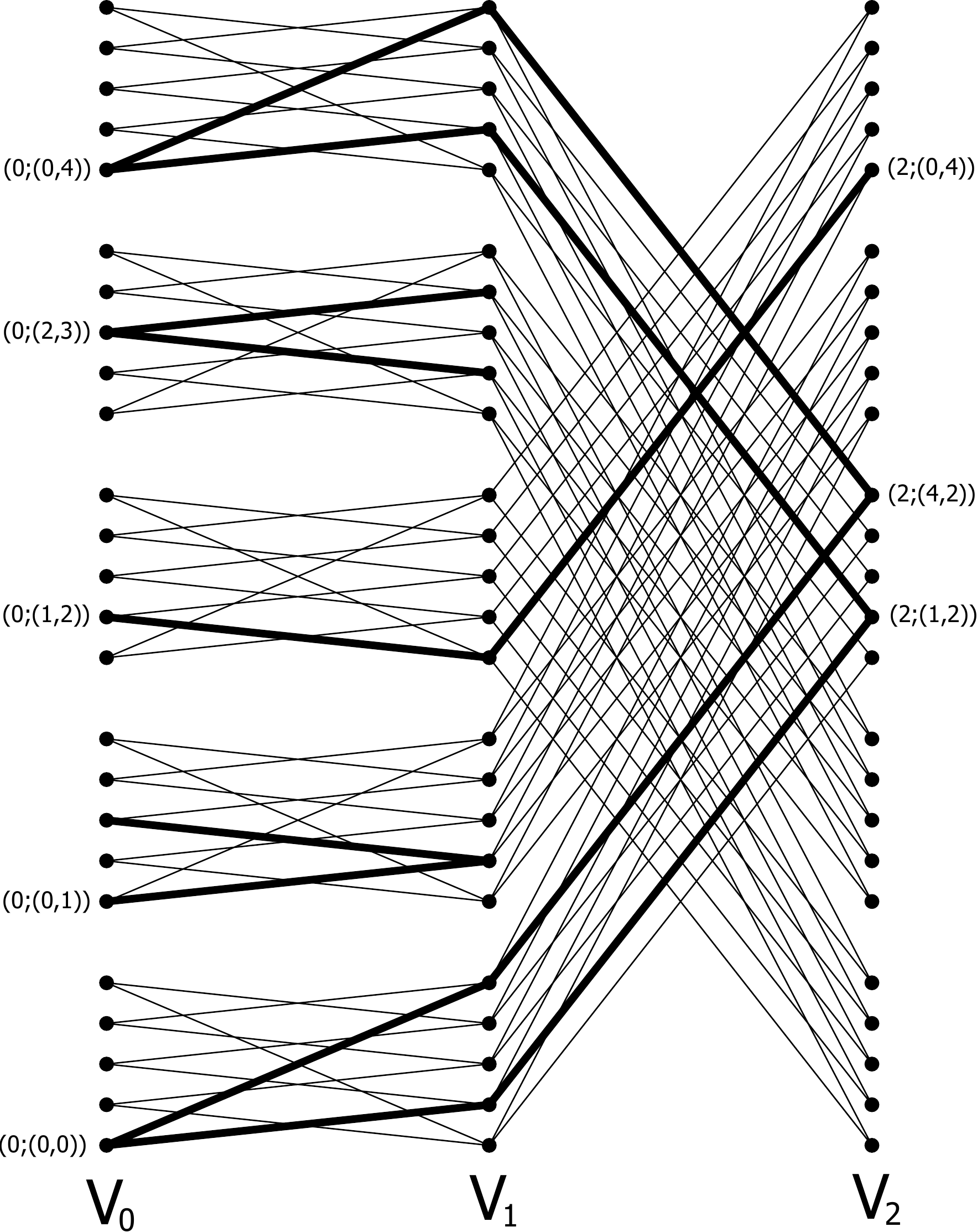}
	\caption{A generic $8$-cycle, two anchors and a non-anchor in the graph $\CPM(3,2,5;2)$.}
	\label{fig:generic}
\end{center}
\end{figure}
\medskip

Now, let $C$ be a cycle of length $k$ and consider any of the $2k$ closed walks of length $k$ corresponding to $C$. We can assign a code of length $k$ to it in the above described way (since the walk is closed we can also assign a corresponding symbol $a$ or $n$ to the initial vertex). We let the {\em trace} of $C$ be defined as the set of all codes assigned to $C$ in this way. In other words, the trace is the equivalence class of all sequences obtained from any of the corresponding codes by reflections and/or cyclic rotations. We denote the trace of $C$ by any of its members (that is, by any of the corresponding codes). The trace of the generic $8$-cycle from \eqref{eq:generic} is thus $anananan$ ($= nananana$). We follow the convention from~\cite{Mar98} of using exponential notation for the traces containing several consecutive identical symbols. For instance, instead of $aannnnnn$ we simply write $a^2n^6$. Of course, each element of $G$ preserves traces of cycles.
\medskip

Before embarking on a thorough analysis of all possible $8$-cycles of $\G$ (and/or other short cycles) we make a few straightforward but useful observations. For any pair of adjacent vertices $\laa i;\mb{v}\raa$ and $\laa i+1;\mb{v'}\raa$ the $s$-tuples $\mb{v}$ and $\mb{v'}$ differ in exactly one component (in fact in the $\ell$-th, where $\ell$ is $i$ mod $s$). Therefore, if we start at some vertex $\laa i, \mb{v}\raa$ of a given cycle $C$ of $\G$ and then traverse $C$ in one of the two possible directions, we change exactly one component of the corresponding $s$-tuple at each step. It is thus clear that the existence of a certain cycle in $\G$ imposes $s$ conditions (congruence equations modulo $n$) on the parameters $n$ and $r$ (some of which may be trivial), one for each of the $s$ components. To make arguments in which we use the above observation easier to explain we introduce the following additional terminology. Let $\laa i;\mb{v}\raa \laa i+1;\mb{v'}\raa$ be an edge of $\G$ where $\ell$ is $i$ mod $s$. Then this edge is said to have {\em label} $\ell$. Therefore, the label of the edge determines the component of the $s$-tuple which will change upon traversal of this edge. Moreover, if $\mb{v'} = \mb{v} + r^i\mb{e_\ell}$ then this edge is said to be {\em positive}, while if $\mb{v'} = \mb{v} - r^i\mb{e_\ell}$ it is said to be {\em negative}.

\begin{lemma}
\label{le:obser}
Let $m, s, n$ and $r$ be as in Assumption~\ref{as:1}. Let $C$ be any cycle of the graph $\G = \CPM(m,s,n;r)$. If we disregard the non-anchors then the positive and negative anchors alternate on $C$. In particular, $C$ has an even number of anchors. Moreover, for each label $\ell \in \{0,1,\ldots , s-1\}$ the cycle $C$ contains either zero or at least two edges of label $\ell$. In addition, if it contains just two edges of label $\ell$, these two edges cannot be consecutive on $C$. 
\end{lemma}

\begin{proof}
The alternation of anchors is clear from the definition. Suppose now that the cycle $C$ contains just one edge of a given label. By Proposition~\ref{pro:HATgroup} we can assume this label is $0$. But then the remarks preceding the statement of this lemma imply that the condition for label $0$ is $1 \equiv 0 \pmod{n}$, which contradicts $n \geq 3$. Similarly, if $C$ contains just two edges of a given label and these edges are consecutive on $C$ then the fact that $s \geq 2$ implies that the corresponding $2$-path is an anchor, and so Proposition~\ref{pro:HATgroup} and the above remarks imply that $2 \equiv 0 \pmod{n}$, which also contradicts $n \geq 3$. 
\end{proof}

Let $C$ be a cycle of $\G$. By Lemma~\ref{le:obser} its trace has an even number of anchors, say $2t$. We assign a nonnegative integer, called the {\em disbalancedness}, to $C$ (and at the same time to its trace) as follows. If $t = 0$ then the disbalancedness is simply the length of $C$. Otherwise, write the trace as $an^{k_1}an^{k_2}\cdots an^{k_{2t}}$, where of course some of the $k_j$'s can be $0$. The corresponding disbalancedness is then the absolute value of the sum $\sum_{j=1}^{2t}(-1)^j k_j$. For instance, the disbalancedness of the generic $8$-cycles (which have trace $anananan$) is $0$, while the disbalancedness of trace $a^2n^6$ (and the corresponding $8$-cycles) is $6$. It is easy to verify that this parameter does not depend on the particular representative code of the trace of $C$, and is thus well defined. Finally, we say that a cycle is {\em coiled} if its disbalancedness is not $0$, and is {\em non-coiled} otherwise. 

\begin{lemma}
\label{le:coiled_ms}
Let $m, s, n$ and $r$ be as in Assumption~\ref{as:1} and let $\G = \CPM(m,s,n;r)$. Then for each coiled cycle $C$ of $\G$ its disbalancedness is a multiple of $ms$.
\end{lemma}

\begin{proof}
Let $C$ be a coiled cycle and let $d$ be its disbalancedness. Pick a vertex $x$ of $C$ and let $0 \leq i \leq ms-1$ be such that $x \in V_i$, where $V_i$ is as in \eqref{eq:V_i}. Then traversing $C$ all the way from $x$ back to $x$ in one of the two possible directions clearly takes us from $V_i$ to $V_{i+d}$ or to $V_{i-d}$. Therefore, $i \pm d \equiv i \pmod{ms}$, which completes the proof.
\end{proof}

\section{The $6$-cycles}
\label{sec:6cycles}

In this section we classify the $2$-arc-transitive $\CPM$ graphs possessing $6$-cycles. We will later show that these are in fact the only $2$-arc-transitive $\CPM$ graphs (see Theorem~\ref{the:2ATclass}). We first determine the possible traces of $6$-cycles in $\CPM$ graphs.

\begin{lemma}
\label{le:6traces}
Let $m, s, n$ and $r$ be as in Assumption~\ref{as:1}. Then the only possible traces of $6$-cycles of the graph $\CPM(m,s,n;r)$ are $a^6$, $an^2an^2$ and $n^6$. 
\end{lemma}

\begin{proof}
By Lemma~\ref{le:obser} any $6$-cycle has an even number of anchors, and so the potential traces for $6$-cycles in decreasing order on the number of anchors are $a^6$, $a^4n^2$, $a^3nan$, $a^2na^2n$, $a^2n^4$, $an^3an$, $an^2an^2$ and $n^6$. As $ms \geq 3$, Lemma~\ref{le:coiled_ms} excludes $a^4n^2$, $a^2na^2n$ and $an^3an$, while Lemma~\ref{le:obser} excludes $a^3nan$. To complete the proof suppose $6$-cycles of trace $a^2n^4$ exist and let $C$ be one of them. By Lemma~\ref{le:coiled_ms} it follows that $ms = 4$ and then Lemma~\ref{le:obser} forces $m = s = 2$. By Proposition~\ref{pro:HATgroup} we can assume that $C$ contains the path $(\laa 0;(0,0)\raa, \laa 1;(1,0)\raa, \laa 0;(2,0)\raa, \laa 1;(3,0)\raa, \laa 2;(3,r)\raa)$. The remaining vertex of $C$ is then of the form $\laa 3;(3 \pm r^2, r)\raa$, and so the conditions for the two labels are $3 \pm r^2 = 0$ and $r \pm r^3 = 0$, which clearly cannot both hold as $r \in \ZZ_n^*$ and $n \notin \{2,4\}$. 
\end{proof}

\begin{proposition}
\label{pro:6m=n}
Let $m, s, n$ and $r$ be as in Assumption~\ref{as:1}. If $\G = \CPM(m,s,n;r)$ possesses $6$-cycles and is $2$-arc-transitive, then $s = 2$ and $\G$ is isomorphic to a graph from Proposition~\ref{pro:2AT}.
\end{proposition}

\begin{proof}
Suppose $\G$ has no $6$-cycles of trace $an^2an^2$. By Proposition~\ref{pro:blocks} and Lemma~\ref{le:6traces} it thus follows that $\G$ possesses $6$-cycles of traces $a^6$ and $n^6$. For those of trace $a^6$ to exist $6 \equiv 0 \pmod{n}$ must hold, and so Lemma~\ref{le:discon} implies that we can assume $n = 3$. Therefore, since $6$-cycles of trace $an^2an^2$ do not exist, $s \neq 2$. On the other hand, as $6$-cycles of trace $n^6$ do exist, Lemma~\ref{le:coiled_ms} implies that $ms \in \{3,6\}$, and so Lemma~\ref{le:obser} implies that $s = 3$ and $m \in \{1,2\}$. It is easy to verify that in each of $\CPM(1,3,3;1)$ and $\CPM(2,3,3;1)$ every anchor lies on just one $6$-cycle while every non-anchor lies on more than one $6$-cycle, contradicting the fact that $\G$ is $2$-arc-transitive.

Therefore, $\G$ possesses $6$-cycles of trace $an^2an^2$. By Lemma~\ref{le:obser} we must have $s = 2$ (and thus $m \geq 2$) and $2(1\pm r^2) = 0$. Since $n \neq 4$, precisely one of $2(1+r^2) = 0$ and $2(1-r^2)=0$ holds. It is now easy to verify (but see the next two paragraphs) that there are precisely two $6$-cycles of trace $an^2an^2$ through each anchor and precisely two through each non-anchor. 

There are two possibilities concerning $6$-cycles of traces other than $an^2an^2$: such $6$-cycles might exist or not. If $\G$ possesses $6$-cycles of traces other than $an^2an^2$, the above remark about the number of $6$-cycles of trace $an^2an^2$ through each anchor and through each non-anchor implies that $\G$ has to possess $6$-cycles of each of the traces $a^6$ and $n^6$. Then $n \in \{3,6\}$ and $m = 3$ (as $s = 2$). Since $\CPM(3,2,6;1) \cong \CPM(6,2,3;1)$ has no coiled $6$-cycles, it is thus not $2$-arc-transitive. Proposition~\ref{pro:2AT} confirms that $\CPM(3,2,3;1)$ is.

For the rest of the proof we will assume the other possibility, that is that the only $6$-cycles of $\G$ are those of trace $an^2an^2$. Let $\delta \in \{-1,1\}$ be such that $2(1 + \delta r^2) = 0$. Regardless of whether $1 + \delta r^2 = 0$ or $1 + \delta r^2 = n/2$ (in which case $n$ must be divisible by $4$ for $r^2$ to be in $\ZZ_n^*$), we get that $r^4 = (\delta r^2)^2 = 1$. Consider now the $2$-arc $P_1 = (\laa 1;(1,0)\raa,\laa 0;(0,0)\raa, \laa 1;(-1,0)\raa )$. There is a $6$-cycle of $\G$ through $P_1\cdot (\laa 2;(-1,r)\raa )$ (it contains $\laa 3;(-1-\delta r^2, r)\raa$ and $\laa 2;(1,r)\raa$) and a $6$-cycle through $P_1\cdot (\laa 2;(-1,-r)\raa )$ (it contains $\laa 3;(-1-\delta r^2, -r)\raa$ and $\laa 2;(1,-r)\raa$), but there is no $6$-cycle of $\G$ through $P_1\cdot (\laa 0;(-2,0)\raa)$. Since $\G$ is $2$-arc-transitive, this shows that for each $2$-arc $(x,y,z)$ of $\G$ there is a unique neighbor $w$ of $z$, different from $y$, such that there is no $6$-cycle through $(x,y,z,w)$ in $\G$. We call $w$ the {\em successor} of $(x,y,z)$. For instance, the successor of the above $P_1$ is $\laa 0;(-2,0)\raa$, while letting $P_2 = (\laa 0;(0,0)\raa, \laa 1;(1,0)\raa, \laa 2;(1,r)\raa )$ the fact that there is a $6$-cycle through $P_2 \cdot (\laa 1;(1,2r)\raa )$ and $P_2 \cdot (\laa 3;(1 + \delta r^2, r)\raa )$, but none through $P_2 \cdot (\laa 3;(1 - \delta r^2, r)\raa )$, implies that the successor of $P_2$ is $\laa 3;(1 - \delta r^2, r)\raa$.

Starting at a given $2$-arc $P$ and extending it by always appending the successor of the current terminal $2$-arc, we obtain the {\em distinguished closed walk} at $P$. Since $\G$ is $2$-arc-transitive, all distinguished closed walks are in the same $\Aut(\G)$-orbit and are thus of the same length. The distinguished closed walk at $P_1$ is clearly the corresponding $G$-alternating cycle (where $G$ is as in Proposition~\ref{pro:HATgroup}), and is thus of length $2n$ or $n$, depending on whether $n$ is odd or even, respectively. We now consider the distinguished closed walk $W$ at $P_2$. The next two vertices on it are $\laa 3;(1-\delta r^2, r)\raa $ and $\laa 4;(1-\delta r^2, r(1-\delta r^2))\raa $. We have two cases to consider. \smallskip

\noindent
{\sc Case 1:} $n$ is odd.\\
In this case $1 + \delta r^2 = 0$, and so Proposition~\ref{pro:iso2} implies that we can assume $r = 1$ (and thus $\delta = -1$). The above two vertices of $W$ are thus $\laa 3;(2,1)\raa$ and $\laa 4;(2,2)\raa$. The length of $W$ is therefore $\mathrm{lcm}(2n,2m) = 2\mathrm{lcm}(n,m)$. For this to be $2n$ we require that $m \mid n$. Since $\G$ is $2$-arc-transitive, Proposition~\ref{pro:blocks} implies that there exists some $\alpha \in \Aut(\G)$, fixing each of $\laa 0;(0,0)\raa$, $\laa 1;(1,0)\raa$ and $\laa -1; (0,-1)\raa$, while interchanging $\laa -1;(0,1)\raa$ with $\laa 1;(-1,0)\raa$. Then $\alpha$ fixes $W$ pointwise and thus also the $(2m)$-th vertex on $W$ (after $\laa 0;(0,0)\raa$), which is $\laa 0;(m,m)\raa$. Consider now the distinguished closed walk $W'$ at $(\laa 0;(0,0)\raa, \laa 2m-1;(0,1)\raa, \laa 2m-2;(1,1)\raa )$. As above we find that the $(2m)$-th vertex on it is $\laa 0;(m,m)\raa$ which we already know is left fixed by $\alpha$. But since $\alpha$ clearly has to reflect $W'$ with respect to $\laa 0;(0,0)\raa $, this implies that the length of this distinguished closed walk is either $2m$ or $4m$. Since this equals $2n$ and $n$ is odd, we must have that $n = m$. Since $s = 2$, $m = n$, $n$ is odd and we can assume that $r = 1$ (by Proposition~\ref{pro:iso2}), this shows that $\G$ is isomorphic to a graph from Proposition~\ref{pro:2AT}, as claimed.\smallskip

\noindent
{\sc Case 2:} $n$ is even.\\
Since $r^4 = 1$, the walk $W$ is of length $d = \mathrm{lcm}(4n/\mathrm{gcd}(n,1- \delta r^2), 2m)$, and so $d = n$ implies that $4 \mid \mathrm{gcd}(n,1- \delta r^2)$. In particular, $4 \mid n$, and so $1+\delta r^2 \neq 0$, implying that $1+\delta r^2 = n/2$. But then $1 - \delta r^2 = 2 + n/2$, and so the fact that $4 \mid \mathrm{gcd}(n,1- \delta r^2)$ implies that $n \equiv 4 \pmod{8}$. As $r^2 \neq \pm 1$ while $r^4 = 1$, the assumption $r^{2m} = \pm 1$ forces $m$ to be even. Also, as $n/2 \equiv 2 \pmod 4$ and $r \equiv \pm 1 \pmod{4}$, it follows that $\delta = 1$, that is $1+r^2 = n/2$. This also forces $d = \mathrm{lcm}(n,2m)$, and so $n = d$ implies $2m \mid n$. It follows that $m = 2k$ for some odd integer $k$. Consider again the distinguished closed walk $W$. As the fourth vertex on it is $\laa 4;(n/2+2, r(n/2+2))\raa$, we see that its $(2m)$-th vertex is $\laa 0;(k(n/2+2), kr(n/2 + 2))\raa = \laa 0 ; (n/2 + m, n/2 + mr)\raa$. This time let $\alpha \in \Aut(\G)$ be an automorphism fixing each of $\laa 0;(0,0)\raa$, $\laa -1; (0,n/2-r)\raa$ and $\laa 1;(1,0)\raa$, while interchanging $\laa 1;(-1,0)\raa$ with $\laa -1;(0,n/2+r)\raa$. Then $\alpha$ fixes $W$ pointwise and thus also $\laa 0;(n/2+m, n/2 + mr)\raa$. Letting $W'$ be the distinguished closed walk at $\laa -1;(0,n/2+r)\raa, \laa 0;(0,0)\raa, \laa 1;(-1,0)\raa )$, we see that $\alpha$ reflects $W'$ with respect to $\laa 0;(0,0)\raa$, and thus interchanges $\laa 0;(n/2-m, n/2 - mr)\raa$ with $\laa 0;(n/2+m, n/2 + mr)\raa$. As the latter is left fixed by $\alpha$, we find that $2m \equiv 0 \pmod {n}$, and so $n = 2m$. Since $s = 2$, $n = 2m$ with $n \equiv 4 \pmod{8}$ and $1+r^2 = n/2 = m$, the graph $\G$ is as in Proposition~\ref{pro:2AT}, as claimed.
\end{proof}

\begin{corollary}
\label{cor:ms=3}
Let $m, s, n$ and $r$ be as in Assumption~\ref{as:1} and let $\G = \CPM(m,s,n;r)$. If $ms = 3$, then $\G$ is arc-transitive but is not $2$-arc-transitive. 
\end{corollary}

\begin{proof}
Since $s \geq 2$, we must have that $m = 1$ and $s = 3$. That $\G$ is arc-transitive now follows from Proposition~\ref{pro:AT}. Since 
$$
	(\laa 0;\mb{0}\raa, \laa 1;\mb{e_0}\raa, \laa 2;\mb{e_0}+r\mb{e_1}\raa, \laa 0;\mb{e_0}+r\mb{e_1}+r^2\mb{e_2}\raa, \laa 1;r\mb{e_1}+r^2\mb{e_2}\raa, \laa 2;r^2\mb{e_2}\raa)
$$ 
is a $6$-cycle of $\G$, Proposition~\ref{pro:6m=n} implies that $\G$ is not $2$-arc-transitive.
\end{proof}

\section{The $8$-cycles}
\label{sec:8cycles}

In this section we prove that the graphs from Proposition~\ref{pro:2AT} are in fact the only $2$-arc-transitive $\CPM$ graphs satisfying Assumption~\ref{as:1} (see Theorem~\ref{the:2ATclass}). We accomplish this by analyzing the possible $8$-cycles.

Let us first note that among the $\CPM$ graphs {\it not}  satisfying Assumption~\ref{as:1}, there are some 2-arc-transitive graphs:   In the case $m= 1, s=2$ the graph, as we have remarked, is toroidal.  Only one of these,  $\{4,4\}_{4,0} =\CPM(2,2,4,1)$ is 2-arc-transitive.

\begin{lemma}
\label{le:8traces}
Let $m, s, n$ and $r$ be as in Assumption~\ref{as:1}. Suppose $\G = \CPM(m,s,n;r)$ is $2$-arc-transitive. Then the only possible traces of $8$-cycles of $\G$ are $T_1 = anananan $, $T_2 = a^8 $, $T_3 = a^3n^2an^2$, $T_4 = an^3an^3$, $T_{12} = a^2n^6$ and $T_{13} = n^8$. Moreover, if $\G$ has $8$-cycles of trace $T_{12}$, then $\G \cong \CPM(3,2,3;1)$.
\end{lemma}

\begin{proof}
We  begin our analysis of $8$-cycles  by noting that the Assumption requires $ms$ to be at least 3, and so, by Lemma~\ref{le:coiled_ms}, traces with disbalancedness exactly 2 can be eliminated. 

The remaining possibilities are  $T_1 = anananan $, $T_2 = a^8 $, $T_3 = a^3n^2an^2$, $T_4 = an^3an^3$, $T_5 = a^5nan$, $T_6 = a^2nan^2an$, $T_7 = a^3na^3n$, all non-coiled; $T_8 = a^4n^4$, $T_9 = a^2n^2a^2n^2$, $T_{10} = a^2na^2n^3$, $T_{11} = an^5an$, all of  disbalancedness 4;  $T_{12} = a^2n^6$ and $T_{13} = n^8$, the final two having disbalancedness 6 and 8, respectively.  The non-coiled cycles are illustrated in Figure~\ref{fig:8cyc_noncoiled}, and those of disbalancedness 4 are shown in   Figure~\ref{fig:8cyc_coiled_4}.   In each figure, the ovals represent the sets $V_i$.

\begin{center}

\begin{figure}[h]
\begin{subfigure}{.3\textwidth}
\begin{center}
\includegraphics[height=20mm]{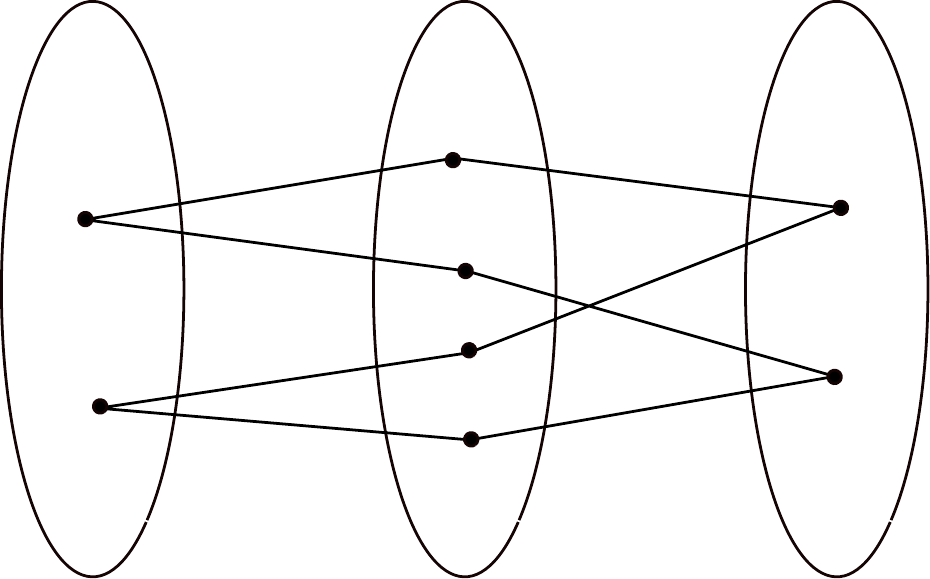}
\caption{$T_1 = anananan$}\label{T1}
\end{center}\end{subfigure}
~
\begin{subfigure}{.3\textwidth}
\begin{center}
\includegraphics[height=20mm]{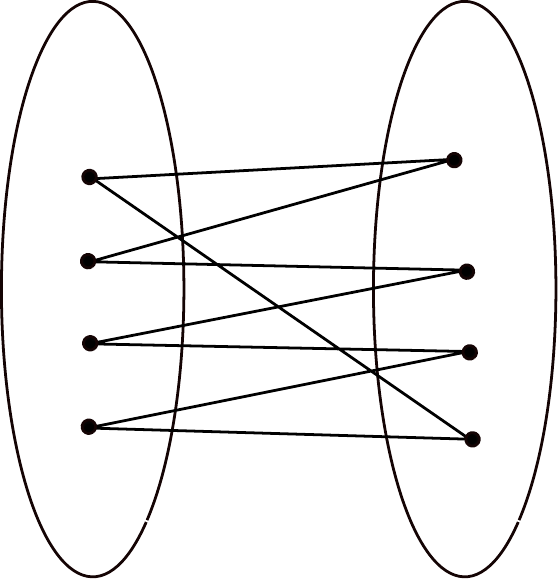}
\caption{$T_2 = a^8$}\label{T2}
\end{center}
\end{subfigure}
~
\begin{subfigure}{.4\textwidth}
\begin{center}
\includegraphics[height=20mm]{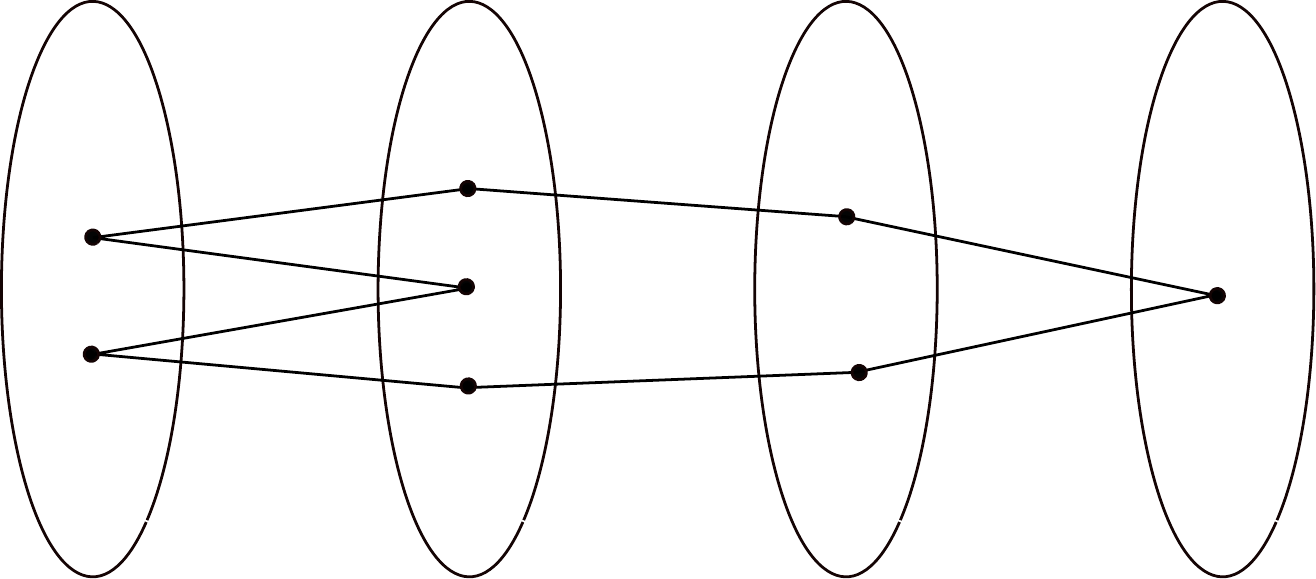}
\caption{$T_3 = a^3n^2an^2$}\label{T3}
\end{center}
\end{subfigure}

\begin{subfigure}{.35\textwidth}
\begin{center}
\includegraphics[height=20mm]{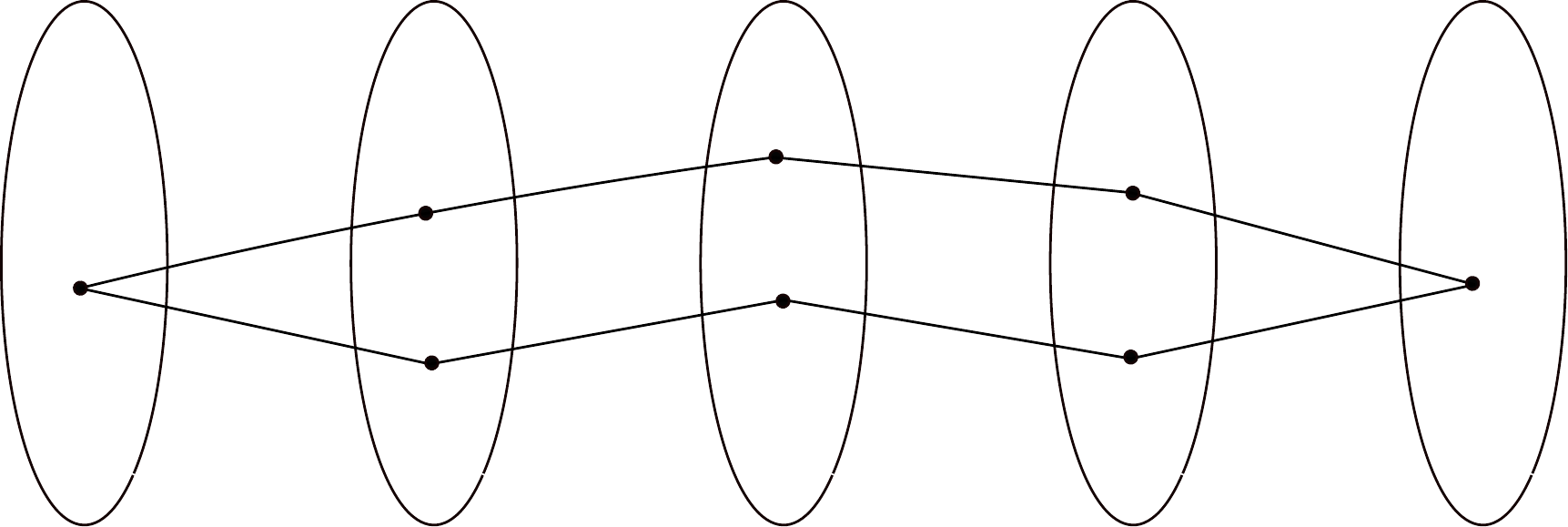}
\caption{$T_4 =  an^3an^3$}\label{T4}
\end{center}
\end{subfigure}
~
\begin{subfigure}{.3\textwidth}
\begin{center}
\includegraphics[height=20mm]{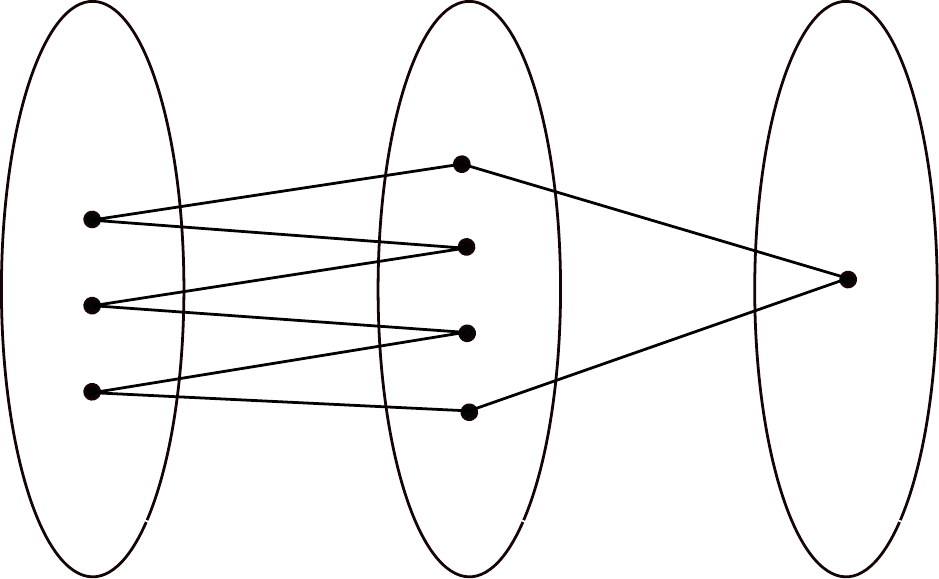}
\caption{$T_5 = a^5nan$}\label{T5}
\end{center}
\end{subfigure}
~
\begin{subfigure}{.3\textwidth}
\begin{center}
\includegraphics[height=20mm]{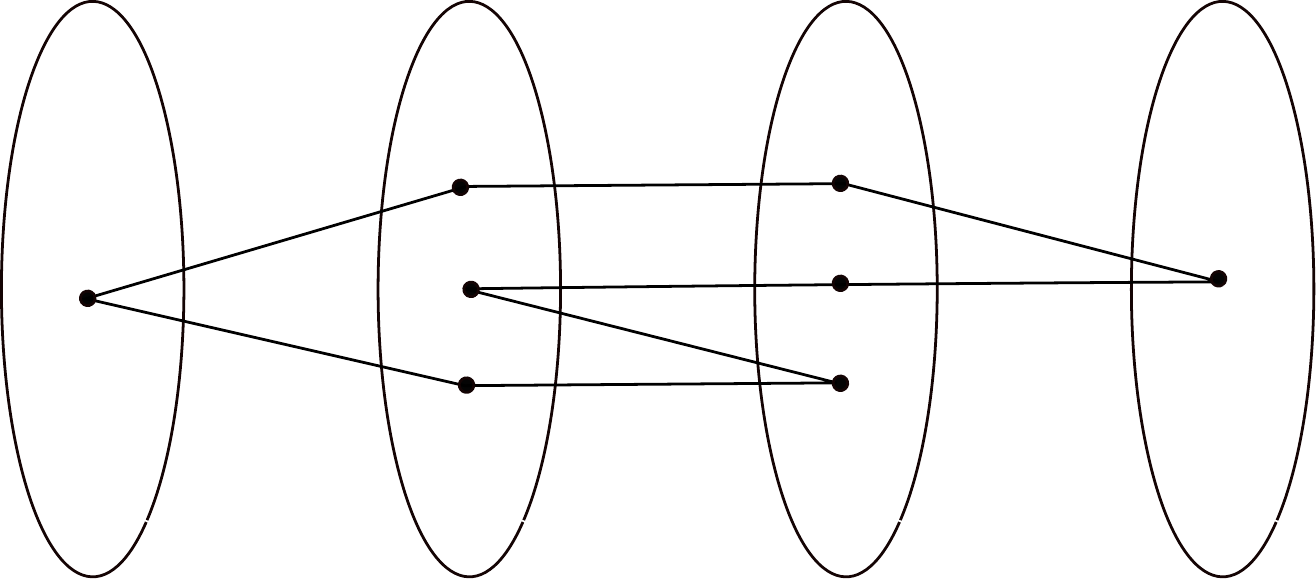}
\caption{$T_6 = a^2nan^2an$}\label{T6}
\end{center}
\end{subfigure}

\begin{subfigure}{\textwidth}
\begin{center}
\includegraphics[height=20mm]{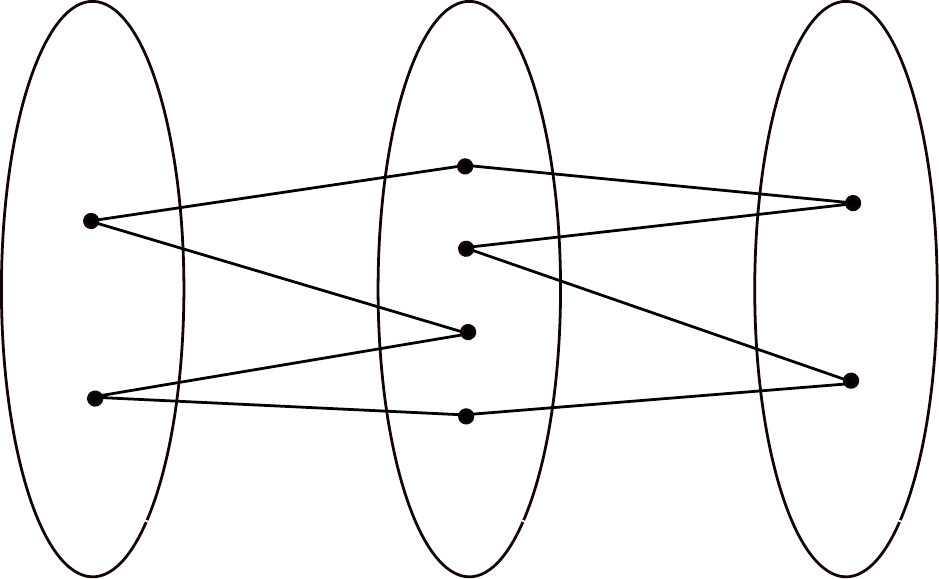}
\caption{$T_7 = a^3na^3n$}\label{T7}
\end{center}
\end{subfigure}

	\caption{Potential non-coiled $8$-cycles.}
	\label{fig:8cyc_noncoiled}

\end{figure}

\end{center}

\begin{center}

\begin{figure}[h]
\begin{subfigure}{.22\textwidth}
\begin{center}
\includegraphics[height=37mm]{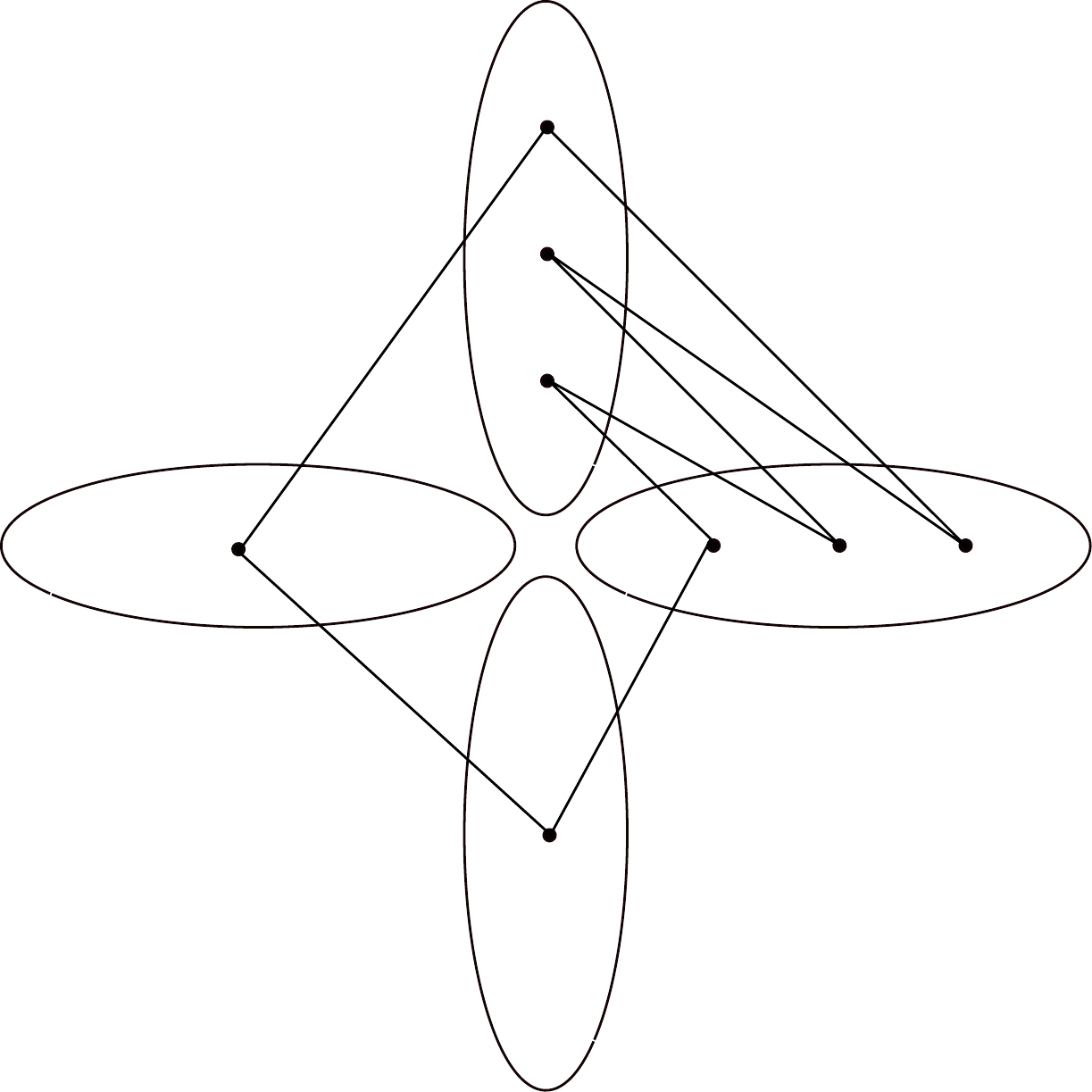}
\caption{$T_8 = a^4n^4$}\label{T8}
\end{center}
\end{subfigure}
~
\begin{subfigure}{.22\textwidth}
\begin{center}
\includegraphics[height=37mm]{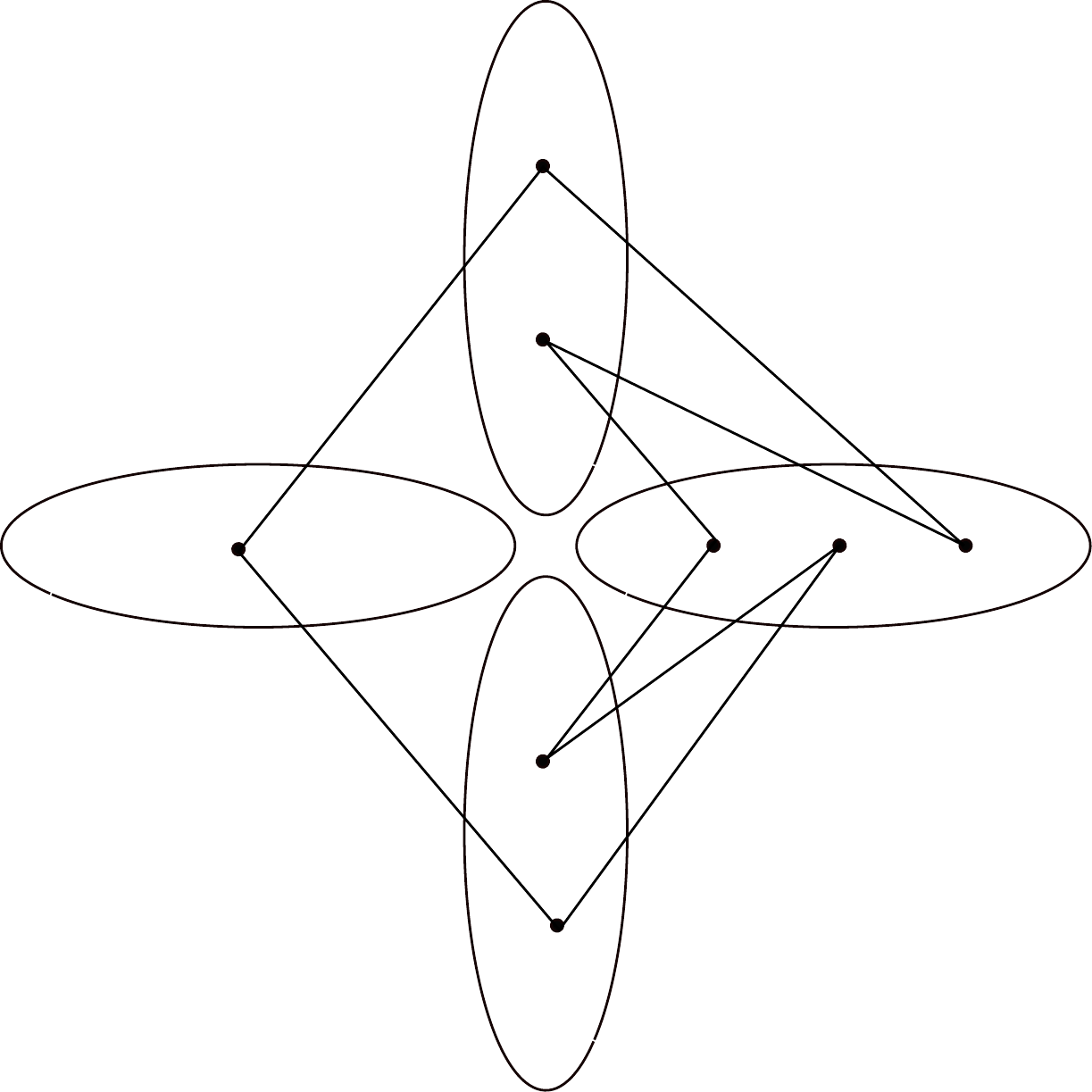}
\caption{$T_{9} = a^2na^2n^3$}\label{T9}
\end{center}
\end{subfigure}
~
\begin{subfigure}{.22\textwidth}
\begin{center}
\includegraphics[height=37mm]{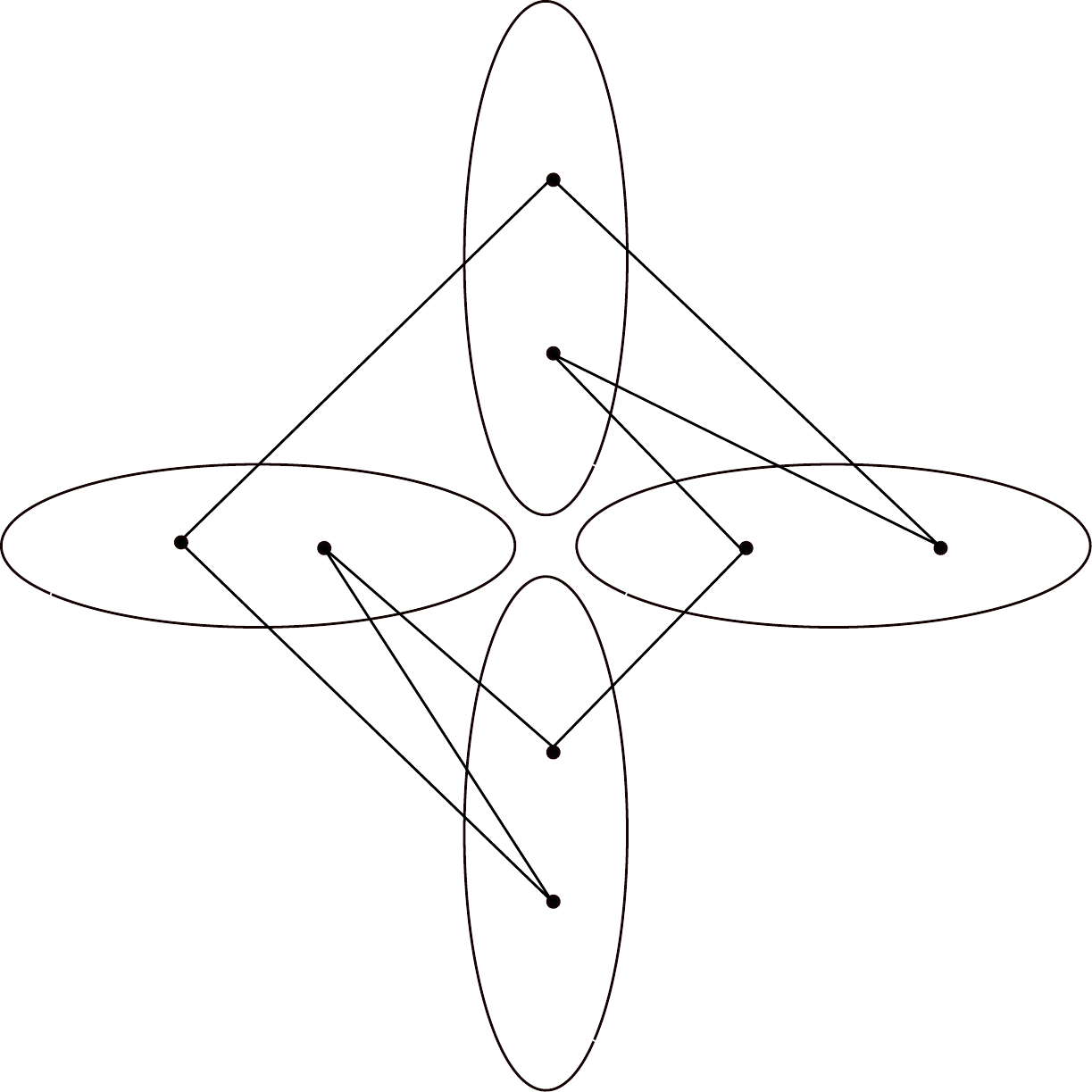}
\caption{$T_{10} = a^2n^2a^2n^2$}\label{T10}
\end{center}
\end{subfigure}
~
\begin{subfigure}{.22\textwidth}
\begin{center}
\includegraphics[height=37mm]{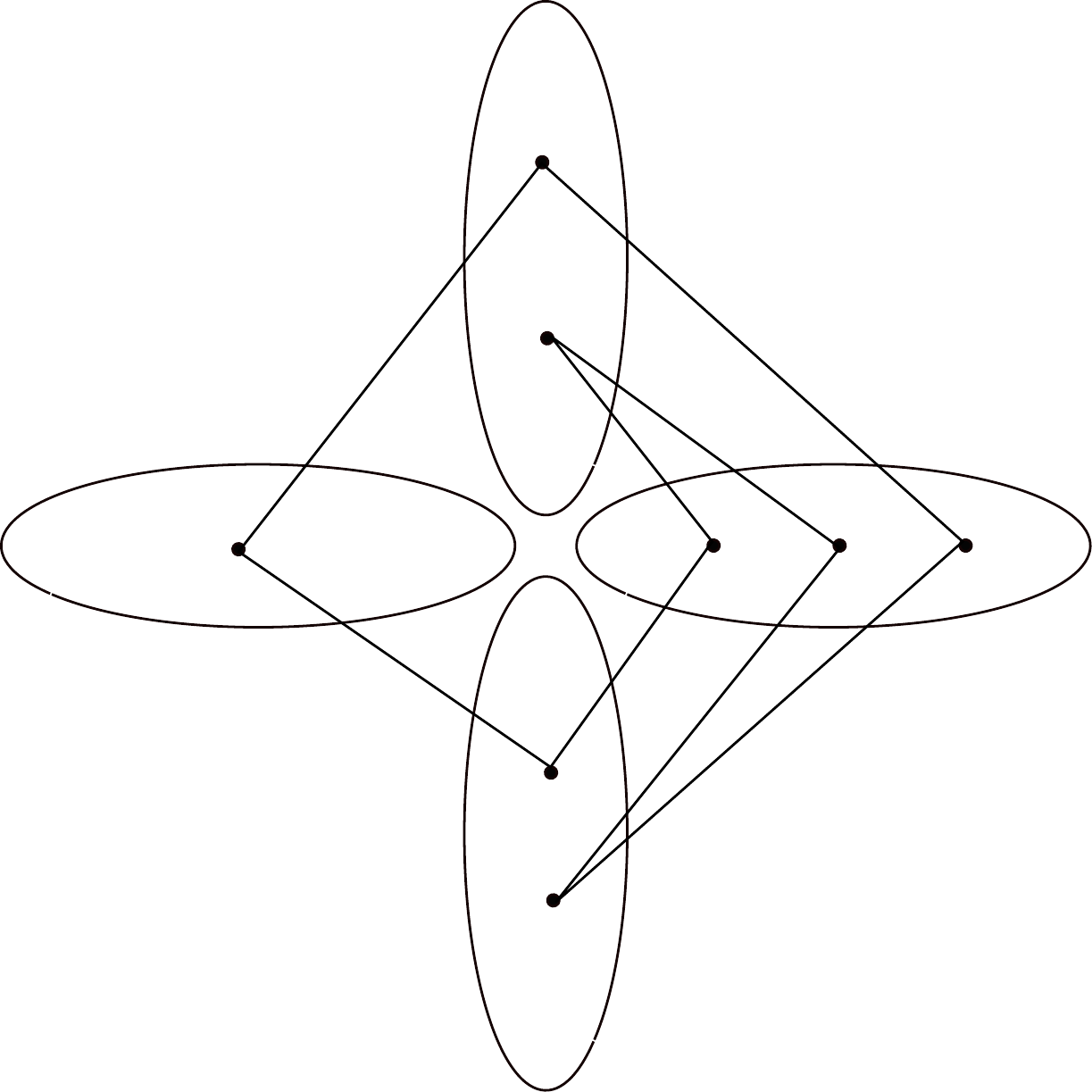}
\caption{$T_{11} = an^5an$}\label{T11}
\end{center}
\end{subfigure}

	\caption{Potential coiled $8$-cycles of disbalancedness $4$.}
	\label{fig:8cyc_coiled_4}

\end{figure}

\end{center}

Lemma~\ref{le:obser} eliminates $T_5 = a^5nan$ (Figure \ref{T5}), while the assumption $n \neq 4$ eliminates $T_6= a^2nan^2an$   (Figure~\ref{T6}) and  $T_7 = a^3na^3n$ (Figure \ref{T7}). 

Suppose now that $C$ is a coiled $8$-cycle of any of the traces with disbalancedness $4$, that is, the trace of $C$ is one of $T_8= a^4n^4$, $T_9 = a^2na^2n^3$, $T_{10} = a^2n^2a^2n^2$ and $T_{11} = an^5an$ (see Figure~\ref{fig:8cyc_coiled_4}). By Lemma~\ref{le:coiled_ms} we must have $ms = 4$, and then Lemma~\ref{le:obser} implies that $s \neq 4$, and so $m = s = 2$. Since $n \neq 4$ and $ms \neq 2$, no anchor of $\G$ lies on a $4$-cycle, and so $2$-arc-transitivity of $\G$ implies that $\G$ has no $4$-cycles. Therefore, $1 \pm r^2 \neq 0$ (otherwise $(\laa 0;(0,0)\raa, \laa 1;(1,0)\raa, \laa 2;(1,r)\raa, \laa 3;(0,r)\raa )$ is a $4$-cycle). This eliminates the traces $T_8$ and $T_{10}$. If $C$ is of trace $T_9$, then the condition for each of the two labels is of the form $3 \pm r^2 = 0$ (by Proposition~\ref{pro:HATgroup} we can assume that $C$ has precisely one vertex from $V_3$). As $r^{ms} = r^4 = \pm 1$, this implies that $9 \equiv \pm 1 \pmod{n}$, and so $n$ divides $8$ or $10$, that is, $n \in \{5, 8, 10\}$ (recall that $n \neq 4$). But then there is no $r \in \ZZ_n^*$ with $r^2 = \pm 3$, a contradiction. Finally, if $C$ is of trace $T_{11}$, then the conditions for the two labels are either of the form $1 \pm r^2 = 0$ or of the form $3 \pm r^2 = 0$. As we already know that none of these is possible, this finally shows that there are no coiled $8$-cycles with disbalancedness $4$ in $\G$.

Suppose finally, that $C$ is of trace $T_{12} = a^2n^6$. By Corollary~\ref{cor:ms=3}, $ms \geq 4$,and then Lemma~\ref{le:coiled_ms} forces $ms = 6$, and so Lemma~\ref{le:obser} implies that $s \in \{2,3\}$. By Proposition~\ref{pro:HATgroup} we can assume that $C$ contains the path 
$$
	(\laa 0;\mb{0}\raa, \laa 1;\mb{e_0}\raa, \laa 0;2\mb{e_0}\raa, \laa 1;3\mb{e_0}\raa, \laa 2;3\mb{e_0}+r\mb{e_1}\raa ).
$$
Therefore, if $s = 3$ the conditions for labels $1$ and $2$ are of the form $1 \pm r^3 = 0$, while the condition for label $0$ is of the form $3 \pm r^3 = 0$, contradicting $n \neq 4$. Therefore, $s = 2$ and $m = 3$, and the conditions for the two labels are 
\begin{equation}
\label{eq:a^2n^6}
3 + \delta_2 r^2 + \delta_4 r^4 = 0 \quad \text{and} \quad 1 +\delta_2' r^2 + \delta_4' r^4 = 0
\end{equation}
for some $\delta_2, \delta_2', \delta_4, \delta_4' \in \{-1,1\}$. Observe that this implies that $n$ is odd (since $r \in \ZZ_n^*$). Now, if $\delta_2\delta_2' = \delta_4\delta_4'$, then adding or subtracting the two equations from \eqref{eq:a^2n^6} yields  $n\mid 4$, which was assumed not to hold. Therefore, subtracting the two equations from \eqref{eq:a^2n^6} yields one of $2 = \pm 2 r^2$ and $2 = \pm 2r^4$. As $n$ is odd and $r^6 = \pm 1$, either of these implies that $r^2 = \pm 1$, and so \eqref{eq:a^2n^6} forces $n = 3$, as claimed. 
\end{proof}

We now analyze the non-coiled $8$-cycles more thoroughly. In particular, we determine the necessary and sufficient conditions for existence of non-coiled $8$-cycles of each of the possible traces from Lemma~\ref{le:8traces} and determine the number of such cycles through any given anchor and through any given non-anchor.  

\begin{proposition}
\label{pro:noncoiled8}
Let $m, s, n$ and $r$ be as in Assumption~\ref{as:1}. Suppose $\G = \CPM(m,s,n;r)$ is $2$-arc-transitive but is not isomorphic to $\CPM(3,2,3;1)$ or $\CPM(5,2,5;1)$. Then for each of the possible traces of non-coiled $8$-cycles of $\G$ the necessary and sufficient condition for the existence of such $8$-cycles, together with the number of such $8$-cycles through each anchor and the number of such $8$-cycles through each non-anchor, are as given in Table~\ref{tab:8noncoiled}.
\end{proposition}

\begin{table}[h]
\begin{center}
\begin{tabular}{|l|c|c|c|}
	 \hline 
	trace & condition & anchor & non-anchor \\ \hline \hline
	$T_1 = anananan$ & none & $2$ & 1 \\ \hline
	$T_2 =a^8$ & $n = 8$ & $1$ & $0$ \\ \hline	
	$T_3 =a^3n^2an^2$ & $s = 2$ and $4\pm 2r^2 = 0$ or $2 \pm 4r^2 = 0$ & 4 & 2 \\ \hline
	$T_4 =an^3an^3$ & $s = 2$ and $2(1 \pm r^2) = 0$ & $2$ & $3$ \\ \hline
	 \end{tabular}       
\end{center}
\caption{Information on non-coiled $8$-cycles in $2$-arc-transitive $\CPM(m,s,n;r)$ graphs.}
\label{tab:8noncoiled}
\end{table}

\begin{proof}
By Corollary~\ref{cor:ms=3} we have that $ms \geq 4$ and by Lemma~\ref{le:8traces} the only possible traces of non-coiled $8$-cycles of $\G$ are  $T_1$,  $T_2$, $T_3$  and $T_4$. We analyze these one by one. 
\smallskip

\noindent
{\sc Trace} $T_1 =anananan$:\\
Let $C$ be an $8$-cycle of this trace. Since $s \geq 2$, Proposition~\ref{pro:HATgroup} implies that we can assume that $C$ contains the non-anchor $(\laa 0;\mb{0}\raa, \laa 1 ; \mb{e_0}\raa, \laa 2 ; \mb{e_0} + r\mb{e_1}\raa)$. In view of the trace of $C$ it thus follows that $C$ in fact contains the path $(\laa 1;-\mb{e_0}\raa, \laa 0;\mb{0}\raa, \laa 1 ; \mb{e_0}\raa, \laa 2 ; \mb{e_0} + r\mb{e_1}\raa, \laa 1 ; \mb{e_0} + 2r\mb{e_1}\raa)$. Since $n \neq 4$, the only possibility to get a valid condition for label 0 is to require that the next two vertices must be $\laa 0 ; 2r\mb{e_1}\raa$ and $\laa 1 ; -\mb{e_0} + 2r\mb{e_1}\raa$. A similar argument for label $1$ shows that $C$ must be a generic $8$-cycle. Therefore, the only $8$-cycles of trace $anananan$ are the generic ones. The above argument also explains the claims about the numbers of such $8$-cycles through a given anchor and through a given non-anchor. 
\smallskip

\noindent
{\sc Trace} $T_2 = a^8$:\\
Clearly, $8$-cycles of this trace exist if and only if $n = 8$ (recall that $n \neq 4$). Moreover, in this case there is precisely one such $8$-cycle through each anchor and there are of course no such $8$-cycles through non-anchors. 
\smallskip

\noindent
{\sc Trace} $T_3 = a^3n^2an^2$:\\
Lemma~\ref{le:obser} implies that $8$-cycles of this trace can exist only if $s = 2$. In that case there are two essentially different possibilities for such $8$-cycles with respect to how the positive and negative anchors are distributed on them. One possibility is that the anchor, surrounded by two non-anchors on each side, is a negative anchor, and the other is that it is a positive anchor . We consider the possibility that it is positive  (as in the case depicted on Figure~\ref{T3}) and leave the other case to the reader (but one may simply apply the isomorphism from Lemma~\ref{le:iso}). 

Let $C$ be such an $8$-cycle. Combining together Proposition~\ref{pro:HATgroup} and Lemma~\ref{le:obser} we can assume that $C$ contains the path $(\laa 1;(-1,0)\raa, \laa 0;(0,0)\raa, \laa 1 ; (1,0)\raa, \laa 0 ; (2,0)\raa, \laa 1 ; (3,0)\raa, \laa 2 ; (3,r)\raa)$. Since we only have two edges of label $1$ and both of them are in $\Gamma[V1; V2]$, choosing one of them uniquely determines the other one if the condition for label 1 is to be valid. In particular, the other vertex of $C$ in $V_2$ is $\laa 2; (-1,r)\raa$. For $C$ to be an $8$-cycle we therefore must have that $3 \pm 2r^2 = -1$, that is $4 \pm 2r^2 = 0$. Conversely, if $4 \pm 2r^2 = 0$, we obviously do get an $8$-cycle of trace $a^3n^2an^2$. The analysis of the above mentioned other situation (where the anchor surrounded by the non-anchors on each side is negative) reveals that a necessary and sufficient condition for the existence of such $8$-cycles is $2 \pm 4r^2 = 0$. 

Since $n \neq 4$ and $r \in \ZZ_n^*$, at most one of the conditions $4 \pm 2r^2 = 0$ can hold and the same applies to $2 \pm 4r^2 = 0$. Moreover, if one of $4 \pm 2r^2 = 0$ and one of $2 \pm 4r^2 = 0$ hold simultaneously, then $2 = \pm 4r^2 = 2\cdot(\pm 4) = \pm 8$, and so $n \in \{3,5,6,10\}$. It follows that $r^2 = \pm 1$, and so $\G$ possesses $6$-cycles (recall that $s = 2$). Proposition~\ref{pro:6m=n} thus implies that $n \in \{3,5\}$ and $m = n$. By Proposition~\ref{pro:iso2} we can assume $r = 1$, and so $\G$ is isomorphic to one of $\CPM(3,2,3;1)$ and $\CPM(5,2,5;1)$, which was assumed not to be the case. This shows that at most one of $4 \pm 2r^2 = 0$ and $2 \pm 4r^2 = 0$ holds. Finally, the above analysis clearly shows that there are precisely $4$ corresponding $8$-cycles through a given anchor (we can choose whether the positive anchor is surrounded by anchors or non-anchors and the two edges of label $1$ can either be both positive or both negative) and precisely $2$ of them through a given non-anchor. 
\smallskip

\noindent
{\sc Trace} $T_4 = an^3an^3$:\\
Lemma~\ref{le:obser} implies that if such $8$-cycles exist, $s \leq 3$ must hold. We first show that in fact $s = 2$. Suppose on the contrary that $s=3$ (in which case $ms \geq 4$ implies that $ms \geq 6$). Let $C$ be an $8$-cycle of trace $T_4$. By Proposition~\ref{pro:HATgroup} we can assume that $C$ contains the path $(\laa 3;\mb{e_0}+r\mb{e_1}+r^2\mb{e_2}\raa, \laa 2;\mb{e_0}+r\mb{e_1}\raa, \laa 1;\mb{e_0}\raa, \laa 0;\mb{0}\raa, \laa 1;-\mb{e_0}\raa)$. For labels $1$ and $2$ to yield a valid condition, we thus require that the next two vertices are $\laa 2;-\mb{e_0}+r\mb{e_1}\raa$ and $\laa 3;-\mb{e_0}+r\mb{e_1}+r^2\mb{e_2}\raa$. For $C$ to be an $8$-cycle we therefore require that $2(1\pm r^3) = 0$. This also shows that if such $8$-cycles exist, there are precisely $4$ such $8$-cycles through each anchor (for each of the labels $1$ and $2$ we can choose whether the corresponding edges are both positive or both negative) and precisely $6$ through each non-anchor. Now, since $s = 3$, there are no $8$-cycles of trace $T_3 = a^3n^2an^2$ in $\G$, while Lemma~\ref{le:8traces} implies that the only possible coiled $8$-cycles are those of trace $T_{13} = n^8$. Therefore, since $\G$ surely has generic $8$-cycles, the only way that each anchor is contained on the same number of $8$-cycles as is each non-anchor, is if $8$-cycles of trace $T_2 = a^8$ exist and there are no coiled $8$-cycles. Therefore, $n = 8$ and each anchor and each non-anchor of $\G$ lies on precisely seven $8$-cycles. 

Now, observe that of the seven $8$-cycles through the $2$-arc $(\laa 1;-\mb{e_0}\raa, \laa 0;\mb{0}\raa, \laa 1;\mb{e_0}\raa)$, precisely one (the one of trace $T_2 =a^8$) goes through the neighbor $\laa 0;2\mb{e_0}\raa$ of $\laa 1;\mb{e_0}\raa$, while there are 3 (one generic and two of trace $T_4 =an^3an^3$) through each of $\laa 2;\mb{e_0}\pm r\mb{e_1}\raa$. On the other hand, of the seven $8$-cycles through the $2$-arc $(\laa 0;\mb{0}\raa, \laa 1;\mb{e_0}\raa, \laa 2;\mb{e_0} + r\mb{e_1}\raa)$, 3 go through the neighbor $\laa 1;\mb{e_0} + 2r\mb{e_1}\raa)$, while there are 2 through each of $\laa 3;\mb{e_0} + r\mb{e_1} \pm   r^2\mb{e_2}\raa$. This contradicts the fact that $\G$ is $2$-arc-transitive, proving that $s \neq 3$. 

Finally, consider the possibility $s = 2$. A similar analysis as above now shows that the necessary and sufficient condition for the corresponding $8$-cycles to exist is $2(1 \pm r^2) = 0$. Moreover, for any of the two labels $\ell \in \{0,1\}$, on each such cycle the two edges of label $\ell$ that do not form an anchor, must be such that one is positive and the other is negative. Since $n \neq 4$, at most one of $2(1 + r^2) = 0$ and $2(1 - r^2) = 0$ can hold, and so there are precisely 2 such $8$-cycles through any anchor and precisely 3 such $8$-cycles through any non-anchor. 
\end{proof}

We are now finally ready to classify the $2$-arc-transitive $\CPM$ graphs.

\begin{theorem}
\label{the:2ATclass}
Let $m, s, n$ be integers with $s \geq 2$ and $n \geq 3$, and let $r \in \ZZ_n^*$ be such that $r^{ms} = \pm 1$. Then the graph $\G = \CPM(m,s,n;r)$ is $2$-arc-transitive if and only if $s = 2$ and one of the following holds:
\begin{itemize}
\itemsep = 0pt
\item[(i)] $n = 4$ and $m = 1$, in which case $\G \cong \CPM(1,2,4;1) \cong \CPM(2,2,4;1)$;
\item[(ii)] $n = m$, $n$ is odd and $r^2 = \pm 1$;
\item[(iii)] $n = 2m$, $m \equiv 2 \pmod {4}$ and $1+r^2 \equiv m \pmod{n}$.
\end{itemize}
\end{theorem}

\begin{proof}
That the graphs from all three items of the theorem are indeed $2$-arc-transitive, follows from Lemma~\ref{le:discon}, Proposition~\ref{pro:iso2} and Proposition~\ref{pro:2AT}. 

To prove the converse, suppose $\G$ is $2$-arc-transitive. By way of contradiction assume that none of the conditions from the three items of the theorem hold. By Lemma~\ref{le:ms=2}, Proposition~\ref{pro:n=4} and Corollary~\ref{cor:ms=3} we can assume that $n \neq 4$ and $ms \geq 4$. If $s = 2$ and $2(r^2 \pm 1) = 0$, then the graph $\G$ possesses $6$-cycles (of trace $an^2an^2$), which by Proposition~\ref{pro:6m=n} contradicts our assumption. It thus follows that $s \geq 3$ or $2(r^2 \pm 1) \neq 0$. In particular, $\G$ has no $8$-cycles of trace $T_4=an^3an^3$.

From this and Proposition \ref{pro:noncoiled8},   we have that the only possibilities for the trace of an 8-cycle in $\Gamma$ are $T_1, T_2, T_3$ and $T_{13} = n^8$. For the purposes of this proof we let $t_c$ be the number of coiled $8$-cycles (those of trace $T_{13}$) through any given non-anchor; further, for $i = 1, 2, 3$, we let $a_i$ be the number of cycles of trace $T_i$ through a given anchor and $b_i$ be the number of cycles of trace $T_i$ through a given non-anchor.  Let $A =  a_1+a_2+a_3$ and $B= b_1+b_2+b_3$. By 2-arc-transitivity, we must have $A = B+t_c$.

 We proceed by proving a series of claims.

\smallskip

\noindent
{\sc Claim 1:} $1 \leq t_c \leq 3$.\\
First notice that  the conditions for $T_2$ and $T_3$ are incompatible: the condition $n= 8$ forces $r^2$ to be 1, making all four of $4\pm 2r^2 = 0$, $2 \pm 4r^2 = 0$ impossible.
Then we see from Table \ref{tab:8noncoiled} that the possibilities for $(A, B)$ and the existence of non-coiled 8-cycles are: (2,1)  [only $T_1$ occurs], (3,1) [$T_1$ and $T_2$], (6,3) [$T_1$ and $T_3$].  To make $A = B+t_c$, then, we must have $t_c = 1, 2$, or 3, respectively.

\smallskip

\noindent
{\sc Claim 2:} $m = 4$, $s = 2$ and $r^4 \pm 1 \neq 0$.\\
By Claim~1 coiled $8$-cycles exist, and so Lemma~\ref{le:coiled_ms} forces $ms \in \{4,8\}$. If $s \geq 4$, then the remark following Proposition~\ref{pro:HATgroup} implies that each non-anchor lies on at least $2(s-2) \geq 4$ different coiled $8$-cycles, contradicting Claim~1. Therefore, $s = 2$. Suppose $r^4 \pm 1 = 0$ and consider the $2$-arc $P = (\laa 0;(0,0)\raa, \laa 1;(1,0)\raa, \laa 2;(1,r)\raa )$. For any $\delta_2, \delta_3 \in \{-1,1\}$ there is thus a coiled $8$-cycle of $\G$ through 
$$
P \cdot (\laa 3;(1+\delta_2 r^2, r)\raa, \laa 4;(1+\delta_2 r^2, r + \delta_3 r^3)\raa),
$$
showing that $P$ lies on at least four different coiled $8$-cycles. As this contradicts Claim~1, this proves that $r^4 \pm 1 \neq 0$, and so $r^{ms} = \pm 1$ forces $m = 4$.
\smallskip

\noindent
{\sc Claim 3:} $t_c = 1$.\\
As $r^4 \neq 1$, we cannot have $n = 8$, and so Table~\ref{tab:8noncoiled} implies that there are no $8$-cycles of trace $a^8$. Consider again the $2$-arc $P$ from the proof of Claim~2. The number of coiled $8$-cycles through $P$ clearly equals the number of solutions (in $\delta_i \in \{-1,1\}$, $2 \leq i \leq 7$) of the system of equations
\begin{align}
	1+\delta_2r^2+\delta_4r^4+\delta_6r^6 & = 0 \label{eq:coiled8_1} \\
	r+\delta_3r^3+\delta_5r^5+\delta_7r^7 & = 0. \label{eq:coiled8_2}
\end{align}
Observe that, as $r \in \ZZ_n^*$, \eqref{eq:coiled8_1} has as many solutions as does \eqref{eq:coiled8_2}, and so $t_c$ is a perfect square. Claim~1 thus implies that $t_c = 1$. From the proof of Claim 1, we see that the only non-coiled 8-cycles are the generic ones, i.e., the ones of trace $T_1$.
\smallskip

\noindent
{\sc Claim 4:} $r^8 = 1$ and precisely one of $1 + r^2 + r^4 + r^6 = 0$ and $1 - r^2 + r^4 - r^6 = 0$ holds.\\
Write $r^8 = \delta$, where $\delta \in \{-1, 1\}$. If we multiply \eqref{eq:coiled8_1} by $r^2$, then use $r^8 = \delta$, then multiply by $\delta\delta_6$,  we have:
\begin{align}
	1+\delta\delta_6r^2+\delta\delta_2\delta_6r^4+\delta\delta_4\delta_6r^6 & = 0 \label{eq:coiled8_3} 
\end{align}
 Since \eqref{eq:coiled8_1} has a unique solution, we must have $\delta_2 = \delta\delta_6$, then $\delta_4 = \delta\delta_2\delta_6 = \delta(\delta\delta_6)\delta_6 =(\delta)^2(\delta_6)^2 = 1$, and finally $\delta_6 = \delta\delta_4\delta_6$, forcing $\delta = 1$.

Therefore, $\delta_4 = 1$, and so $\delta_2 = \delta_6$. In other words, one of $1 + r^2 + r^4 + r^6 = 0$ and $1 - r^2 + r^4 - r^6 = 0$ holds. Note that by Claim~3 both of these two conditions cannot hold simultaneously.
\smallskip

\noindent
We are now ready to obtain the final contradiction. Let $\epsilon \in \{-1, 1\}$ be such that $1 + \epsilon r^2 + r^4 + \epsilon r^6 = 0$. By Claim~3 and the proof of Claim~1 only coiled and generic $8$-cycles exist in $\G$ and there are precisely two $8$-cycles through any anchor and precisely two $8$-cycles through any non-anchor. Consider again the $2$-arc $P$ from the proof of Claim~2. The unique coiled $8$-cycle through it contains $P\cdot (\laa 3;(1+\epsilon r^2,r)\raa)$, the unique generic $8$-cycle through it contains $P\cdot (\laa 1;(1,2r)\raa)$, while there is no $8$-cycle through $P\cdot (\laa 3;(1-\epsilon r^2,r)\raa)$. As in the proof of Proposition~\ref{pro:6m=n} we can thus speak of {\em distinguished closed walks} at a given $2$-arc (we keep adding the unique neighbor of the current last vertex such that the corresponding terminal $3$-arc does not lie on any $8$-cycle of $\G$). The initial section of the distinguished closed walk at $P$ is thus
$$
P \cdot (\laa 3;(1- \epsilon r^2,r)\raa, \laa 4;(1- \epsilon r^2,r(1- \epsilon r^2))\raa, \ldots , \laa 0;(1- \epsilon r^2+r^4- \epsilon r^6,r(1- \epsilon r^2+r^4- \epsilon r^6))\raa).
$$
This walk is clearly of length $d = 8(n/\mathrm{gcd}(n,1- \epsilon r^2+r^4- \epsilon r^6))$, which is thus divisible by $8$. Since $\G$ is $2$-arc-transitive, the same must hold for the length of the distinguished closed walk at $P' = (\laa 0;(0,0)\raa, \laa 1;(1,0)\raa, \laa 0;(2,0)\raa)$. Since no $8$-cycle of $\G$ contains two consecutive anchors, the distinguished closed walk at $P'$ is thus the corresponding $G$-alternating cycle (where $G$ is as in Proposition~\ref{pro:HATgroup}). It is thus of length $2n$ or $n$, depending on whether $n$ is odd or even, respectively. As this equals $d$ (which is divisible by $8$), $n$ must be even, and so $d = n$ and $\mathrm{gcd}(n, 1 - \epsilon r^2 + r^4 - \epsilon r^6) = 8$. Since this implies $8 \mid n$, we must have $1 - \epsilon r^2 + r^4 - \epsilon r^6 \equiv 0 \pmod {8}$. But as $1 + \epsilon r^2 + r^4 + \epsilon r^6 = 0$, we also have that $1 + \epsilon r^2 + r^4 + \epsilon r^6 \equiv 0 \pmod{8}$, and so $2 + 2r^4 \equiv 0 \pmod{8}$, which is clearly impossible as $r^2 \equiv 1 \pmod {8}$. 
\end{proof}

\section{The automorphism group of the $\CPM$ graphs}
\label{sec:AutGr}

In this section we determine the automorphism groups of all $\CPM$ graphs which in turn enables us to determine which of the $\CPM$ graphs are half-arc-transitive. We first focus on the most symmetric examples, namely the $2$-arc-transitive ones. 

\begin{theorem}
\label{the:Aut2AT}
Let $m, s, n$ be integers with $s \geq 2$ and $n \geq 3$, and let $r \in \ZZ_n^*$ be such that $r^{ms} = \pm 1$. If the graph $\G = \CPM(m,s,n;r)$ is $2$-arc-transitive, then the vertex-stabilizers in $\Aut(\G)$ are of order $24$ and the automorphism group $\Aut(\G)$ is generated by the group $G$ from Proposition~\ref{pro:HATgroup} and $\nu$ from the proof of Proposition~\ref{pro:2AT} or the restriction of $\nuall$ to $\G$.
\end{theorem}

\begin{proof}
By Theorem~\ref{the:2ATclass}, $s = 2$ and we can assume that one of items (ii) and (iii) from that theorem holds. In particular, $2(1 \pm r^2) = 0$, and so $\G$ possesses $6$-cycles of trace $an^2an^2$. That $\CPM(2,2,4;1)$ and $\CPM(3,2,3;1)$ have the properties stated in this theorem, can be verified directly by a computer. For the rest of the proof we can thus assume that $n \geq 5$. The first part of the proof of Proposition~\ref{pro:6m=n} thus implies that the only $6$-cycles of $\G$ are those of trace $an^2an^2$ and that there are precisely two $6$-cycles through any anchor of $\G$ and precisely two $6$-cycles through any non-anchor of $\G$. It also shows that we have certain distinguished closed walks in $\G$.

By Proposition~\ref{pro:blocks} the first part of the proof will be complete if we can show that the only automorphism of $\G$ fixing a vertex and all of its neighbors is the identity. Since $\G$ is connected and $2$-arc-transitive, it suffices to prove that for some vertex $x$ of $\G$ and some neighbor $y$ of $x$, each automorphism of $\G$ fixing $x$ and all of its neighbors also fixes all neighbors of $y$. Let $\alpha \in \Aut(\G)$ be an automorphism fixing the vertex $\laa 1;(1,0)\raa$ and all of its neighbors. We show that then $\alpha$ also fixes all neighbors of $\laa 2;(1,r)\raa$. Let $\delta \in \{-1,1\}$ be such that $2(1+\delta r^2) = 0$. The proof of Proposition~\ref{pro:6m=n} shows that the distinguished closed walk at $P = (\laa 0;(0,0)\raa, \laa 1;(1,0)\raa; \laa 2;(1,r)\raa)$ continues through $\laa 3;(1-\delta r^2, r)\raa$. Since $P$ is left fixed pointwise by $\alpha$, $\laa 3;(1-\delta r^2, r)\raa$ must also be fixed by $\alpha$. Similarly, the distinguished closed walk at $P' = (\laa 0;(2,0)\raa, \laa 1;(1,0)\raa, \laa 2;(1,r)\raa)$ continues through $\laa 3;(1+\delta r^2,r)\raa$, and as $P'$ is left fixed pointwise by $\alpha$, $\laa 3;(1+\delta r^2,r)\raa$ is also fixed by $\alpha$. Thus, $\alpha$ fixes each neighbor of $\laa 2;(1,r)\raa$.

To complete the proof recall that the group $G$ from Proposition~\ref{pro:HATgroup} acts vertex-transitively. It is easy to see that the restrictions of $\tauall_0$, $\tauall_1$ and $\nuall$ to $\G$ (or $\nu$) suffice to get all $24$ elements in the stabilizer of the vertex $\laa 0;(0,0)\raa$.
\end{proof}

The next results determine the automorphism group of the non-2-arc-transitive $\CPM$ graphs and classify the half-arc-transitive examples.

\begin{proposition}
\label{pro:not2AT_HATgroup}
Let $m, s, n$ and $r$ be as in Assumption~\ref{as:1}, let $\G = \CPM(m,s,n;r)$ and let $G$ be as in Proposition~\ref{pro:HATgroup}. If $\G$ is not $2$-arc-transitive, then one of the following holds:
\begin{itemize}\itemsep = 0pt
\item[(i)] $\G$ is half-arc-transitive and $\Aut(\G) = G$;
\item[(ii)] $\G$ is arc-transitive and $[\Aut(\G) : G] = 2$.
\end{itemize}
\end{proposition}

\begin{proof}
By Proposition~\ref{pro:blocks} the sets $V_i$ from \eqref{eq:V_i} are blocks of imprimitivity for $\Aut(\G)$. Orient the edge $\laa 0;\mb{0}\raa \laa 1;\mb{e_0}\raa$ from $\laa 0;\mb{0}\raa$ to $\laa 1;\mb{e_0}\raa$ and then use the action of the half-arc-transitive group $G$ to orient all other edges of $\G$. Each edge in $\G[V_i, V_{i+1}]$ is then oriented from $V_i$ to $V_{i+1}$. Denote by $\vec{\G}$ the corresponding oriented graph and let $A = \Aut(\vec{\G})$. In other words, $A$ is the subgroup of $\Aut(\G)$ consisting of all automorphisms of $\G$ preserving the above chosen orientation. Of course, $G \leq A$. We claim that in fact $A = G$, which then implies (i) and (ii) from the proposition. To shorten notation in this proof and the proof of Theorem~\ref{the:HAT} denote 
\begin{equation}
\label{eq:f}
\mb{f} = \mb{e_0} + r\mb{e_1}+r^2\mb{e_2} + \cdots + r^{s-2}\mb{e_{s-2}} + r^{s-1}\mb{e_{s-1}}.
\end{equation}

By way of contradiction assume $A \neq G$. By the remark following Proposition~\ref{pro:HATgroup} the group $G$ acts $s$-arc-transitively on $\vec{\G}$, and so the fact that $G \lneq A$ implies that $A$ acts $(s+1)$-arc-transitively on $\vec{\G}$. There thus exists an automorphism $\alpha \in A$, fixing pointwise the $s$-arc 
\begin{equation}
\label{eq:s_arc}
	(\laa 0;\mb{0}\raa, \laa 1;\mb{e_0}\raa, \laa 2;\mb{e_0}+r\mb{e_1}\raa, \ldots , \laa s;\mb{f}\raa),
\end{equation}
and swapping $\laa s+1;\mb{f} + r^s\mb{e_0}\raa$ and $\laa s+1;\mb{f} -r^s\mb{e_0}\raa$. Since the sets $V_i$ are blocks of imprimitivity for $\Aut(\G)$ (and thus also for $A$), $\alpha$ fixes each one of them setwise. Therefore, 
$$
\laa s;\mb{f}-2r^s\mb{e_0}\raa \alpha = \laa s;\mb{f} + 2r^s\mb{e_0}\raa.
$$
Since $\alpha$ fixes both $\laa 0;\mb{0}\raa$ and $\laa 1;\mb{e_0}\raa$, it fixes pointwise the whole $G$-alternating cycle containing the corresponding edge, and so it must also fix $\laa 1;(1-2r^s)\mb{e_0}\raa$. Note that there is a walk of length $s-1$ from $\laa 1;(1-2r^s)\mb{e_0}\raa$ to $\laa s;\mb{f}-2r^s\mb{e_0}\raa$ going through $V_2, V_3, \ldots , V_{s-1}$. There should thus also exist a walk of length $s-1$ from $\laa 1;(1-2r^s)\mb{e_0}\raa$ to $\laa s;\mb{f}+2r^s\mb{e_0}\raa$ going through $V_1, V_2, \ldots , V_{s-1}$. But this clearly implies $-2r^s = 2r^s$, contradicting $n \neq 4$. Therefore, $A = G$, as claimed. 
\end{proof}

\begin{theorem}
\label{the:HAT}
Let $m, s, n$ be integers with $s \geq 2$ and $n \geq 3$, and let $r \in \ZZ_n^*$ be such that $r^{ms} \pm 1 = 0$. Then the graph $\CPM(m,s,n;r)$ is half-arc-transitive if and only if $2(r^{2s} \pm 1) \neq 0$. Moreover, in this case its automorphism group coincides with the group $G$ from Proposition~\ref{pro:HATgroup} and has vertex-stabilizers isomorphic to the elementary abelian $2$-group of rank $s$.
\end{theorem}

\begin{proof}
By Lemma~\ref{le:ms=2}, Proposition~\ref{pro:n=4} and Corollary~\ref{cor:ms=3} we can assume that $n \neq 4$ and $ms \geq 4$. The second part of the theorem thus follows from Proposition~\ref{pro:HATgroup} and Proposition~\ref{pro:not2AT_HATgroup}. As for the first part, one implication follows from Proposition~\ref{pro:AT}. To complete the proof assume that $\G$ is arc-transitive and let us prove that then $2(r^{2s} \pm 1) = 0$. If $m \leq 2$ this clearly holds, so we can assume that $m \geq 3$ (and consequently $ms \geq 6$). If $\G$ is $2$-arc-transitive, Theorem~\ref{the:2ATclass} ensures that $2(r^{2s} \pm 1) = 0$. 

We are thus left with the possibility that $\G$ is arc-transitive but not $2$-arc-transitive. The idea of the proof is very similar to the one from the proof of Proposition~\ref{pro:not2AT_HATgroup}. By Proposition~\ref{pro:blocks} the sets $V_i$ from \eqref{eq:V_i} are blocks of imprimitivity for $\Aut(\G)$. Choose the $G$-induced orientation of the edges of $\G$ from the proof of Proposition~\ref{pro:not2AT_HATgroup}. Since $\G$ is arc-transitive, Proposition~\ref{pro:HATgroup} implies that there exists $\beta \in \Aut(\G)$ reversing the $s$-arc from \eqref{eq:s_arc}. It follows that the $G$-alternating cycle containing the arc $(\laa 1;\mb{e_0}\raa, \laa 0;\mb{0}\raa)$ is mapped to the $G$-alternating cycle containing the arc $(\laa s-1;\mb{f} - r^{s-1}\mb{e_{s-1}}\raa, \laa s;\mb{f}\raa)$, where $\mb{f}$ is as in \eqref{eq:f}. In particular, 
$$
	\laa 1;(1-2r^s)\mb{e_0}\raa\beta = \laa s-1 ; \mb{f} + (2r^{2s-1} - r^{s-1})\mb{e_{s-1}}\raa.
$$
Of course, the vertex $\laa s;\mb{f}-2r^s\mb{e_0}\raa$ is mapped by $\beta$ to one of $\laa 0;\pm 2r^{ms-1}\mb{e_{s-1}}\raa$. But then $\beta$ has to map the walk of length $s-1$ from $\laa 1;(1-2r^s)\mb{e_0}\raa$ to $\laa s;\mb{f} -2r^s\mb{e_0}\raa$ going through each of $V_2, V_3, \ldots , V_{s-1}$ to a walk of length $s-1$ from $\laa s-1 ; \mb{f} + (2r^{2s-1} - r^{s-1})\mb{e_{s-1}}\raa$ to one of $\laa 0;\pm 2r^{ms-1}\mb{e_{s-1}}\raa$ going through each of $V_1, V_2, \ldots , V_{s-2}$. Therefore, $2r^{2s-1} = \pm 2r^{ms-1}$, and so $2r^{2s} = \pm 2$, as claimed.
\end{proof}

Combining together Proposition~\ref{pro:AT}, Lemma~\ref{le:ms=2}, Proposition~\ref{pro:n=4} and the results of this section thus yields the following corollary.

\begin{corollary}
\label{cor:Aut_AT}
Let $m, s, n$ be integers with $s \geq 2$ and $n \geq 3$, and let $r \in \ZZ_n^*$ be such that $r^{ms} \pm 1 = 0$. Let $\G = \CPM(m,s,n;r)$ and let $G$ be as in Proposition~\ref{pro:HATgroup}. Then $\G$ is arc-transitive if and only if $2(r^{2s} \pm 1) = 0$. Moreover, if $\G$ is arc-transitive but not $2$-arc-transitive then either $\G \cong \CPM(m,s,4;1)$ is a Praeger-Xu graph and has vertex-stabilizers in $\Aut(\G)$ of order $2^{s(m-1)+1}$, or $\Aut(\G)$ is generated by $G$ and the restriction of $\etaall$ or $\etaall'$ from the remark following Proposition~\ref{pro:AT} to $\G$ in which case the vertex-stabilizers in $\Aut(\G)$ are of order $2^{s+1}$.
\end{corollary}

\section{Isomorphisms and concluding remarks}
\label{sec:Iso}

In this final section we consider the possible isomorphisms between different $\CPM$ graphs and give some directions for possible future research. As in the previous section we first consider the $2$-arc-transitive $\CPM$ graphs.

\begin{proposition}
\label{pro:all_iso_2AT}
Let $m,s,n$ and $r$, as well as $m', s', n'$ and $r'$ be as in Assumption~\ref{as:1}. If the graphs $\G = \CPM(m,s,n;r)$ and $\G' = \CPM(m',s',n';r')$ are both $2$-arc-transitive then $\G' \cong \G$ if and only if $m' = m$, $s' = s = 2$, $n' = n$ and $(r'r^{-1})^2 = \pm 1$.
\end{proposition}

\begin{proof}
Suppose $\G$ and $\G'$ are both 2-arc-transitive. By Theorem~\ref{the:2ATclass} we have that $s' = s = 2$. Now, if $m' = m$, $n' = n$ and $(r'r^{-1})^2 = \pm 1$ all hold, then $\G' \cong \G$ by Proposition~\ref{pro:iso2}.

To prove the converse suppose $\G' \cong \G$. By Theorem~\ref{the:2ATclass} either $n$ is odd and $m = n$, in which case $\G$ is of order $2n^3 \equiv 2 \pmod{4}$, or $n \equiv 4 \pmod{8}$ and $m = n/2$, in which case $\G$ is of order $2m^3 \equiv 0 \pmod{4}$. As $\G$ and $\G'$ are of the same order it thus follows that $n$ and $n'$ are of the same parity and consequently also $m' = m$ and $n' = n$. Theorem~\ref{the:2ATclass} also implies that if $n$ is odd then $r^2 = \pm 1$ and (as $n' = n$) $r'^2 = \pm 1$, while if $n$ is even $r^2 = m - 1 = r'^2$. In any case, $r'^2 = \pm r^2$, as claimed.
\end{proof}

In Proposition~\ref{pro:iso2} two sufficient conditions for a pair of $\CPM$ graphs to be isomorphic was given. The next result gives a partial converse to Proposition~\ref{pro:iso2}. The reader will note that, in view of Lemma~\ref{le:discon}, the assumptions on the parity of $n, n', m$ and $m'$ can always be achieved using appropriate isomorphisms.

\begin{proposition}
\label{pro:alliso}
Let $m,s,n$ and $r$, as well as $m', s', n'$ and $r'$ be as in Assumption~\ref{as:1} where in addition we assume that each of $n$ and $n'$ is either odd or divisible by $4$ and if $n$ or $n'$ is even then $m$ or $m'$, respectively, is also even. Suppose that $\CPM(m,s,n;r) \cong \CPM(m',s',n';r')$. Then $m = m'$, $s = s'$, $n = n'$ and one of the following holds:
\begin{itemize}\itemsep = 0pt
\item[(i)] $(rr')^s  = \pm 1$ or $(r^{-1}r')^s = \pm 1$;
\item[(ii)] $n$ is divisible by $4$, $m$ is even, and either $2((rr')^s \pm 1) = 0$ or $2((r^{-1}r')^s \pm 1) = 0$.
\end{itemize}
\end{proposition}

\begin{proof}
Denote $\G = \CPM(m,s,n;r)$ and $\G' = \CPM(m',s',n';r')$. In view of Proposition~\ref{pro:all_iso_2AT} we can assume that $\G$ (and hence also $\G'$) is not $2$-arc-transitive. By Proposition~\ref{pro:blocks} the sets $V_i$ from \eqref{eq:V_i} are blocks of imprimitivity for $\Aut(\G)$ and the sets $V'_i$ are blocks of imprimitivity for $\Aut(\G')$. Let $\Phi \colon \G \to \G'$ be an isomorphism mapping the vertex $\laa 0 ; \mb{0}\raa$ of $\G$ to the vertex $\laa 0 ; \mb{0}\raa$ of $\G'$. We can assume that $\Phi$ also maps $\laa 1 ; \mb{e_0}\raa$ of $\G$ to $\laa 1 ; \mb{e_0} \raa$ of $\G'$ (otherwise replace $r'$ by its inverse and use Lemma~\ref{le:iso}). By Proposition~\ref{pro:blocks} the set of non-anchors of $\G$ is an $\Aut(\G)$-orbit and similarly the set of non-anchors of $\G'$ is an $\Aut(\G')$-orbit. The fact that there are four non-anchors with a given vertex as their mid-vertex but only two anchors with this vertex as their mid-vertex thus implies that $\Psi$ maps non-anchors of $\G$ to non-anchors of $\G'$ and thus also anchors of $\G$ to anchors of $\G'$. This clearly shows that it has to map each set $V_i$ of $\G$ to the set $V'_i$ of $\G'$. 

Now, let $G$ be as in Proposition~\ref{pro:HATgroup} for $\G$ and let $G'$ be the group corresponding to the group $G$ from Proposition~\ref{pro:HATgroup} for $\G'$. The above remarks show that $\Phi$ maps the $G$-alternating cycles of $\G$ to the $G'$-alternating cycles of $\G'$. In particular, they are of the same length. As the $G$-alternating cycles of $\G$ are of length $2n$ or $n$, depending on whether $n$ is odd or even respectively, and $n$ is either odd or divisible by $4$, this length is either divisible by $4$ if and only if $n$ is even. This clearly shows that $n = n'$. Moreover, as $\G$ and $\G'$ must have vertex-stabilizers of the same size, the assumption $n \neq 4$, together with Theorem~\ref{the:HAT} and Corollary~\ref{cor:Aut_AT}, imply that $s = s'$. By assumption $\G$ is of order $msn^s$ or $ms(n/2)^s$, depending on whether $n$ is odd or even, respectively, and so the fact that $\G$ and $\G'$ are of the same order implies $m = m'$. 

To complete the proof we thus only have to show that the parameters $r$ and $r'$ satisfy one of the two conditions from (i) and (ii). We do this by using a similar idea as in the proofs of Proposition~\ref{pro:not2AT_HATgroup} and Theorem~\ref{the:HAT}. In view of Proposition~\ref{pro:HATgroup} and the above assumption on $\Phi$ we can assume that $\Phi$ in fact maps the $s$-arc from \eqref{eq:s_arc} to the $s$-arc
$$
	(\laa 0;\mb{0}\raa, \laa 1;\mb{e_0}\raa, \laa 2,\mb{e_0}+r'\mb{e_1}\raa, \ldots , \laa s;\mb{f}'\raa)
$$
of $\G'$, where $\mb{f}' = \mb{e_0}+r'\mb{e_1}+\cdots + r'^{s-1}\mb{e_{s-1}}$. It thus follows that $\laa 1;(1-2r^s)\mb{e_0}\raa\Phi = \laa 1;(1-2r^s)\mb{e_0}\raa$. But as $\laa s, \mb{f} - 2r^s\mb{e_0}\raa$ has to be mapped to one of $\laa s, \mb{f}' \pm 2r'^s\mb{e_0}\raa$, considering the walk of length $s-1$ from $\laa 1;(1-2r^s)\mb{e_0}\raa$ to $\laa s, \mb{f} - 2r^s\mb{e_0}\raa$, going through each of $V_2, V_3, \ldots , V_{s-1}$, shows that $2r^s = \pm 2r'^s$, and so $2((r'r^{-1})^s \pm 1) = 0$. If $(r'r^{-1})^s = \pm 1$, the proof is complete. We can thus assume that $n$ is even and $(r'r^{-1})^s = n/2 \pm 1$. Set $q = r'r^{-1}$ and note that the fact that $q \in \ZZ_n^*$ implies that $q$ is odd, and so $n$ must be divisible by $4$. Finally, as $q^{ms} = \pm 1$ but $q^s \neq \pm 1$, $m$ must be even, as claimed. 
\end{proof}

In view of Proposition~\ref{pro:iso2} and Proposition~\ref{pro:alliso} the only type of pairs of $\CPM$ graphs whose parameters satisfy the parity conditions from Proposition~\ref{pro:alliso} for which the question of whether the two graphs are isomorphic or not has not been settled, are the ones corresponding to item (ii) of Proposition~\ref{pro:alliso} with $m \equiv 2 \pmod{4}$. The fact that the graphs $\CPM(6,2,52;3)$ and $\CPM(6,2,52;15)$ are not isomorphic (this can be verified by a suitable computer software) suggests that in such a case the two graphs are not isomorphic. We thus pose the following question.

\begin{question}
Let $m \geq 2$, $s \geq 2$ and $n \geq 8$ be integers such that $n$ is divisible by $4$ and $m \equiv 2 \pmod{4}$. Suppose $r, r' \in \ZZ_n^*$ are such that $r^{ms} = \pm 1$, $r'^{ms} = \pm 1$, $(r^{-1}r')^s \neq \pm 1$ and $(rr')^s \neq \pm 1$. Is it true that then the graphs $\CPM(m,s,n;r)$ and $\CPM(m,s,n;r')$ are not isomorphic? 
\end{question}

We conclude the paper by a short discussion of possible future research projects. Using Theorem~\ref{the:HAT}, Proposition~\ref{pro:iso2} and Proposition~\ref{pro:alliso} it is easy to verify that the smallest connected loosely-attached (that is, with $s \geq 2$) half-arc-transitive $\CPM$ graph of even radius is the graph $\CPM(3,2,28;3)$ which is of order $2352$. This shows that none of the 15 tetravalent half-arc-transitive loosely attached graphs with even radius and vertex-stabilizers of order at least $4$ from the Census~\cite{PotSpiVer15} is a $\CPM$ graph. These 15 graphs thus must belong to some other family of such graphs. One could thus try to identify the corresponding family(ies) and classify the half-arc-transitive members as we have done for the $\CPM$ graphs. 

A good starting point for such an investigation is definitely by taking a closer look at the so-called Attebery graphs which were introduced in 2008 by Casey Attebery~\cite{Att08} (but see also~\cite{PotWil??}). These graphs represent a nice generalization of the $\CPM$ graphs containing not just the loosely-attached examples but also graphs with larger attachment numbers. We are convinced that some of our methods for the investigation of symmetries of the graphs in question will carry over to the larger class of Attebery graphs. 

One might also try to extend the results from the paper~\cite{KuzMalPot18} in which all elementary abelian covers over ``doubled cycles'' such that a vertex- and edge-transitive group of automorphisms lifts were characterized. Of course, the class of all corresponding covers contains some of the (half-arc-transitive) $\CPM$ graphs, but some graphs from \cite{KuzMalPot18} can be outside the $\CPM$ family while some $\CPM$ graphs are not elementary abelian covers of ``doubled cycles''. It should also be pointed out that determining which of the covers from~\cite{KuzMalPot18} are indeed half-arc-transitive is not at all an easy task. But perhaps some methods from this paper might be of use.

\section*{Acknowledgements}
All authors acknowledge support by the Slovenian Research Agency bilateral research project BI-US/18-20-075. \v S. Miklavi\v c acknowledges support by the Slovenian Research Agency (research core funding No. P1-0285 and research projects N1-0062, J1-9110, J1-1695, N1-0140). P.~\v Sparl acknowledges support by the Slovenian Research Agency (research core funding No. P1-0285 and research projects J1-9108, J1-9110, J1-1694, J1-1695).


\begin{thebibliography}{}

\bibitem{AlbAlkMutPraSpi16} J.~A.~Al-bar, A.~N.~Al-kenani, N.~M.~Muthana, C.~E.~Praeger, P.~Spiga, 
		Finite edge-transitive oriented graphs of valency four: a global approach, 
		{\em Electron. J. Combin.} {\bf 23} (2016) 1.10.

\bibitem{AlsXu94} B.~Alspach, M.Y.~Xu, 
		$1/2$-arc-transitive graphs of order $3p$, 
		{\em J. Algebraic Combin.} {\bf 3} (1994), 347--355.
		
\bibitem{Att08} C.~Attebery, 
		Constructing Graded Semi-Transitive Orientations of valency $4(p - 1)$,
		M.Sc. Thesis, Northern Arizona University, 2008.
		
\bibitem{ConPotSpa15} M.D.E.~Conder, P.~Poto\v cnik, P.~\v Sparl,
		Some recent discoveries about half-arc-transitive graphs,
		{\em Ars Math. Contemp.} {\bf 8} (2015), 149--162.

\bibitem{GarPra94A} A.~Gardiner, C.~E.~Praeger,
		On $4$-valent symmetric graphs,
		{\em European J. Combin.} {\bf 15} (1994), 375--381.

\bibitem{GarPra94} A.~Gardiner, C.~E.~Praeger,
		A characterization of certain families of 4-valent symmetric graphs,
		{\em European J. Combin.} {\bf 15} (1994), 383--397.
		
\bibitem{JajPotWil19} R.~Jajcay, P.~Poto\v cnik, S.~Wilson,
		The Praeger-Xu graphs: cycle structures, maps, and semitransitive orientations,
		{\em Acta Math. Univ. Comenianae} {\bf 88} (2019), 269--291.
		
\bibitem{KuzMalPot18} B.~Kuzman, A.~Malni\v c, P.~Poto\v cnik,
		Tetravalent vertex- and edge-transitive graphs over doubled cycles,
		{\em J. Comb. Theory, Ser. B} {\bf 131} (2018), 109--137. 
		
\bibitem{Mal98} A.~Malni\v c,
		Group actions, coverings and lifts of automorphisms,
		{\em Discrete Math.} {\bf 182} (1998), 203--218.
		
\bibitem{MalMar99} A.~Malni\v c, D.~Maru\v si\v c, 
		Constructing 4-valent $\frac{1}{2}$-transitive graphs with a nonsolvable group, 
		{\em J. Comb. Theory, Ser. B} {\bf 75} (1999), 46--55.
		
\bibitem{MalNedSko00} A.~Malni\v c, R.~Nedela, and M.~\v Skoviera,
                 Lifting Graph Automorphisms by Voltage Assignments,
                 {\em European~J.~Combin.} {\bf 21} (2000), 927--947.
		
\bibitem{Mar98} D.~Maru\v si\v c, 
		Half-transitive group actions on finite graphs of valency 4, 
		{\em J. Comb. Theory, Ser. B} {\bf 73} (1998), 41--76.

\bibitem{Mar05} D.~Maru\v si\v c, 
		Quartic half-arc-transitive graphs with large vertex stabilizers, 
		{\em Discrete Math.} {\bf 299} (2005), 180--193.
		
\bibitem{MarNed01} D.~Maru\v si\v c, R.~Nedela,
		On the point stabilizers of transitive groups with non-self-paired suborbits of length 2,
		{\em J. Group Theory } {\bf 4}, (2001), 19--43.
		
\bibitem{MarPra99} D. Maru\v si\v c, C.~E.~Praeger, 
		Tetravalent graphs admitting half-transitive group actions: alternating cycles, 
		{\em J. Comb. Theory, Ser. B} {\bf 75} (1999), 188--205.
		
\bibitem{PotSpiVer15} P.~Poto\v cnik, P.~Spiga, G.~Verret,
		A census of 4-valent half-arc-transitive graphs and arc-transitive digraphs of valence two, 
		{\em Ars Math. Contemp.} {\bf 8} (2015), 133--148.
		
\bibitem{PotWil07} P.~Poto\v cnik, S.~Wilson,
		Tetravalent edge-transitive graphs of girth at most 4,
		{\em J. Comb. Theory, Ser. B} {\bf 97} (2007), 217--236.

\bibitem{PotWil??} P.~Poto\v cnik, S.~Wilson,
		Recipes for edge-transitive tetravalent graphs, arXiv:1608.04158.
		
\bibitem{PraXu89} C.~E.~Praeger, M.~Y.~Xu, 
		A characterization of a class of symmetric graphs of twice prime valency, 
		{\em European J. Combin.} {\bf 10} (1989), 91--102.
		
\bibitem{RamSpa17} A.~Ramos Rivera, P.~\v Sparl,
		The classification of half-arc-transitive generalizations of Bouwer graphs,
		{\em European J. Combin.} {\bf 64} (2017), 88--112.
		
\bibitem{RamSpa19} A.~Ramos Rivera, P.~\v Sparl,
		New structural results on tetravalent half-arc-transitive graphs,
		{\em J. Comb. Theory, Ser. B} {\bf 135} (2019), 256--278.
		
\bibitem{Spa08} P.~\v Sparl, 
		A classification of tightly attached half-arc-transitive graphs of valency $4$, 
		{\em J. Comb. Theory, Ser. B} {\bf 98} (2008), 1076--1108.
		
\bibitem{Spa09} P.~\v Sparl, 
		On the classification of quartic half-arc-transitive metacirculants, 
		{\em Discrete Math.} {\bf 309} (2009), 2271--2283.
		
\bibitem{TayXu94} D.~E.~Taylor, M.~Y.~Xu, 
		Vertex-primitive $\frac{1}{2}$-transitive graphs, 
		{\em J. Aust. Math. Soc., Ser. A} {\bf 57} (1994), 113--124.
		
\bibitem{Wil04} S.~Wilson,
		Semi-transitive graphs,
		{\em J. Graph Theory} {\bf 45} (2004), 1--27.
		
\bibitem{Xu92} M.~Y.~Xu, 
		Half-transitive graphs of prime cube order, 
		{\em J. Algebraic Combin.} {\bf 1} (1992), 275--282.
		
\end{thebibliography}
\end{document}